\documentclass{article}

\usepackage{amssymb,amsmath,amsthm,amsfonts,amscd}
\usepackage{bbm}
\usepackage[english]{babel}
\usepackage{graphicx}
\usepackage{color}
\usepackage[all]{xy}
\usepackage{hyperref}
\hypersetup{
    colorlinks=true,
    linkcolor=blue,
    citecolor=blue,
    filecolor=blue,
    urlcolor=blue
}

\numberwithin{equation}{section}

\newtheorem{theorem}{Theorem}[section]
\newtheorem{lemma}{Lemma}[section]
\newtheorem{propposition}{Proposition}[section]
\newtheorem{definition}{Definition}[section]
\newtheorem{observation}{Observation}[section]

\title{ {Hierarchic Controllability for a nonlinear parabolic equation in one dimension}
\author{\sc
m. r. nu\~nez-ch\'avez\,
\thanks{Instituto de Matem\'atica e Estat\'istica,\,UFF, RJ, Brasil,\, miguelnunez@mat.uff.br} \,\,\,\,\&\,\,\,\,
j. l\'imaco\,
\thanks{Instituto de Matem\'atica e Estat\'istica,\,UFF, RJ, Brasil,\, jlimacoferrel@hotmail.com}
}
\date{}
}
\begin{document}

\maketitle

\abstract{This paper deals with the hierarchical control of the parabolic equation.We use Stackelberg–Nash strategies. As usual, we consider one leader and two followers. To each leader we associate a Nash equilibrium corresponding to a bi-objective optimal control problem, then, we look for a leader that solves null controllability e with trajectories problem. We consider linear and nonlinear equations in dimension 1.}\\

\vspace{1cm}


\section{Introduction}\label{s1}

The development of science and technology has motivated many branches of control theory. Initially, in the classical control theory, we encountered problems where a system must reach a predetermined target by the action of a single control, for example, find a control of minimum norm such that the design specifications are met. To the extent that more realistic situations were considered, it was necessary to include several different(and even contradictory) control objectives, as well as develop theory that would handle the concepts of multi-criteria optimization, where optimal decisions need to be taken in the presence of trade-off between these different objectives.\\

There are many points of view to deal with multi-objective problems. Notions of economics and game theory were introduced in the works of J. Nash \cite{Nash}, V. Pareto \cite{Pare} and H. Von Stackelberg \cite{Stack}, where each has a particular philosophy to solve these problems. In the context of the control of PDEs, a relevant question is whether one is able to steer the system to a desired state (exactly or approximately) by applying controls that
correspond to one of these strategies. According to the formulation introduced by H. Von Stackelberg \cite{Stack}, we assume the presence of various local controls, called followers which have their own objectives, and a global control, called leader, with a different goal from the rest of the players (in the case, the followers). The general idea of this strategy is a game of hierarchical nature, where players compete against each other, so that the leader makes the fist move and then followers react optimally to the action of the leader. Since many followers
are present and each has a specific objective, it is intended that these are in Nash equilibrium.\\

The concept of hierarchic control in the context of parabolic PDEs was introduced in \cite{Lions3} when the author analyzed the approximate controllability for a system associated with a parabolic equation. There, he considered one main control (the leader) and an additional secondary controls (the followers). In \cite{Lions3} \cite{GuiGo}, the hierarchic control of a parabolic PDE and the Stokes systems have been analyzed and used to solve an approximate controllability problem.\\

In \cite{Cara1}, \cite{Cara5}, \cite{Cara2} a similar strategy has been used to deduce the exact controllability (to the trajectories) for a parabolic PDE, the problem was analyzed in \cite{Cara1}, \cite{Cara5} with distributed leader and follower controls, while \cite{Cara2} deals with distributed and boundary controls. In this paper, we will analyze a related problem for a nonlinear equation by considering a leader and two different followers. We will apply the Stackelberg–Nash rule, which combines the strategies of cooperative optimization of Stackelberg and the non-cooperative strategy of Nash.

\section{The problem and its results}\label{s2}

Let $I = (0,L) \subset \mathbb{R}$ be a bounded open interval. Let $T > 0$ be given and let us consider the cylinder $Q=I \times (0,T)$, with lateral boundary $\Sigma = \partial I \times (0,T)$. In the sequel, we will denote by $C$ a generic positive constant, sometimes, we will indicate the data on which $C$ depends by writing $C(I), C(I,T)$, etc. The usual norm and scalar product in $L^2(I)$ will be respectively denoted by $\|\cdot\|$ and $(\cdot,\cdot)$. \\

We are interested in proving the null controllability of a multi-objective parabolic PDE problem in $Q$, where we apply the Stackelberg-Nash strategy; we will assume without less of generality that only three controls are applied (one leader and two followers).\\
We will consider the following system
\begin{equation}
\label{EC1ch5}
\left\{
\begin{array}{rl}
&y_t -  (a(y_x,t,x) y_x)_x + F(y,y_x) = f 1_\mathcal{O} + v^1 1_{\mathcal{O}_1} + v^2 1_{\mathcal{O}_2} \ \text{in} \ Q,\\
&y(0,t) = y(L,t) = 0\ \ \text{ in } \ (0,T),\\
&y(0) = y_{0}\ \ \text{in } \ I.
\end{array}
\right.
\end{equation}
In system $(\ref{EC1ch5})$,\  $y$ is the state, the set $\mathcal{O} \subset I$ is the main control domain and $\mathcal{O}_1, \mathcal{O}_2 \subset I$ are the secondary control domain (it is supposed to be small); the controls are $f$, $v^1$ and $v^2$, where $f$ is the leader and $v^1,v^2$ are the followers.

We assume that 
\begin{enumerate}
	\item $a \in C^3(\mathbb{R} \times [0,T] \times \overline{I})$,
	\item There exist positive constants $a_0,a_1$ such that $$a_0 \leq a(s,t,x) \leq a_1,\ \ \forall (s,t,x) \in \mathbb{R} \times [0,T] \times \overline{I},$$
	\item There exists a positive constant $M$ such that
$$\sum_{i=1}^3|D_i a(s,t,x)|+\sum_{i,j=1}^3 |D_{ij}^2 a(s,t,x)| + \sum_{ij,k=1}^3 |D_{ijk}^3 a(s,t,x)| \leq M,\ \ \forall (s,t,x) \in \mathbb{R} \times [0,T] \times \overline{I}$$
	\item $F \in C^{2}(\mathbb{R} \times \mathbb{R})$ with bounded derivatives.
\end{enumerate}   
Let $\mathcal{O}_{1,d}, \mathcal{O}_{2,d}  \subset I$ be open sets, representing observation domains for the followers. We will consider the functional
\begin{equation}
\label{eq1ch5}
J_i(f;v^1,v^2):= \dfrac{\alpha_i}{2} \iint_{\mathcal{O}_{i,d} \times (0,T)} |y-y_{i,d}|^2 dxdt + \dfrac{\mu_i}{2} \iint_{\mathcal{O}_i \times (0,T)} |v^i|^2 dxdt,
\end{equation}
where $\alpha_i, \mu_i >0$ are constants and $y_{i,d} \in L^2(\mathcal{O}_{i,d} \times (0,T))$ are given function.\\
The control process can be described as follows:\\
\textbf{1.} The followers $v^i$ assume that the leader $f$ has made a choice and intend to be a Nash equilibrium for the costs $J_i$. Thus, once $f$ has been fixed, we look for controls $v^i \in L^2(\mathcal{O}_i \times (0,T))$ that satisfy
\begin{equation}
\label{eq2ch5}
J_1(f;v^1,v^2) = \underset{\hat{v}^1}{\text{min}}\ J_1(f;\hat{v}^1,v^2),\ \ J_2(f;v^1,v^2) = \underset{\hat{v}^2}{\text{min}}\ J_2(f;{v}^1,\hat{v}^2) .
\end{equation}
\begin{definition}
\label{def1ch5}
Any pair $(v^1,v^2)$ satisfying $(\ref{eq2ch5})$ is called a Nash equilibrium for $J_1$ and $J_2$.
\end{definition}
Note that, if the functional $J_i\ (i=1,2)$  are convex, then $(v^1,v^2)$ is a Nash equilibrium  if and only if
\begin{equation}
\label{eq3ch5}
J'_1(f;v^1,v^2)(\hat{v}^1,0) = 0,\ \ \ \forall \hat{v}^1 \in L^2(\mathcal{O}_1 \times (0,T))
\end{equation}
and
\begin{equation}
\label{eq4ch5}
J'_2(f;v^1,v^2)(0,\hat{v}^2) = 0,\ \ \ \forall \hat{v}^2 \in L^2(\mathcal{O}_2 \times (0,T)).
\end{equation}
\begin{definition}
\label{def3ch5}
Any pair $(v^1,v^2)$ satisfying $(\ref{eq3ch5})$ and $(\ref{eq4ch5})$ is called a Nash quasi-equilibrium for $J_1$ and $J_2$.
\end{definition}
\textbf{2.} Once the Nash equilibrium has been identified and fixed for each $f$, we look for a control $\hat{f} \in L^2(\mathcal{O} \times (0,T))$ subject to the restriction of null controllability
\begin{equation}
\label{eq5ch5}
y(T)= 0 \ \ \text{in}\ I.
\end{equation}

\subsection{The Main Results}\label{sa2}

Let us study the following problems.
\begin{theorem}
\label{teo1ch5}
Let us assume that 
\begin{equation}
\label{condition1ch5}
\mathcal{O}_{i,d} \cap \mathcal{O} \neq \emptyset,\ \ i=1,2
\end{equation} 
Also, suppose that one of the following two conditions hold
\begin{equation}
\label{condition2ch5}
\mathcal{O}_{1,d}=\mathcal{O}_{2,d}
\end{equation}
or
\begin{equation}
\label{eq8ch5}
\mathcal{O}_{1,d} \cap \mathcal{O} \neq \mathcal{O}_{2,d} \cap \mathcal{O}.
\end{equation}
Then, there exist $\epsilon>0$, $\mu_0>0$ only depending on $I$, $T$, $\mathcal{O}$, $\mathcal{O}_i$, $\mathcal{O}_{i,d}$ and $\alpha_i$  and a positive function $\hat{\rho} = \hat{\rho}(t)$ blowing up at $t=T$ with the following property: if $\mu_i \geq \mu_0$, the $y_{i,d}$ is such that
\begin{equation}
\label{eq12ch5}
\iint_{\mathcal{O}_{i,d} \times (0,T)} \hat{\rho}^2 |y_{i,d}|^2 dxdt < \epsilon,
\end{equation}
there exists $\delta>0$, such that for any $y_0 \in H^3(I) \cap H^1_0(I)$ with $\|y_0\|_{H^3(I)} \leq \delta$, there exist controls $f \in L^2(\mathcal{O} \times (0,T))$ and associated Nash quasi-equilibrium $(v^1,v^2)$ such that the corresponding solutions to $(\ref{EC1ch5})$ satisfy $(\ref{eq5ch5})$.
\end{theorem}
A natural question is whether there are semilinear systems for which the concepts of Nash equilibrium and Nash quasi-equilibrium are equivalent. An answer is given by the following result:
\begin{theorem}
\label{teo2ch5}
Let us assume that $y_{i,d} \in L^\infty(\mathcal{O}_{i,d} \times (0,T))$. Suppose that $y_0 \in H^3(I) \cap H^1_0(I)$\ with\ $\|y_0\|_{H^3(I)}~\leq~\delta$. Then, there exists $C>0$ such that, if $f \in L^2(\mathcal{O} \times (0,T))$ and $\mu_i$ satisfies
\begin{equation*}
\mu_i \geq C(1+\|y^0\|_{H_0^1(I)}+\|f\|_{L^2(\mathcal{O} \times (0,T))}),
\end{equation*}
the pair $(v^1,v^2)$ is a Nash equilibrium for $J_i$ of $(\ref{EC1ch5})$.
\end{theorem}

\textbf{Outline of the paper.} The content of this article are organized as follows. 


In Section \ref{s3}  applying the Nash-Stackelberg strategies, this is, using the adjoint system to characterize the followers control, we will find the (nonlinear and linear) Optimality System.

In Section \ref{s4} we analyze the Nash Quasi-equilibrium for the linear optimality system, thus,  we prove a null controllability result using a standard technique based on an observability estimate. This will be deduced using Carleman inequalities.

In Section \ref{s5}, we analyze the nonlinear case in Theorem \ref{teo1ch5}, for this, we prove a controllability property by using Right Inverse Function Theorem for Banach Spaces techniques. And finally we prove Theorem \ref{teo2ch5}, this is, for $\mu_1$, $\mu_2$ large enough the Nash quasi-equilibrium is a Nash equilibrium. 
 
In Section \ref{s6} we briefly address the case Hierarchical controllability with trajectories, the strategy is very similar to the previous case. 

And in Section \ref{s7} we address additional comments and open questions.

\section{Characterization of Nash Quasi-Equilibrium}\label{s3}

Note that the convexity of the functional $J_i$ are not guaranteed. For this reason, we must re-define the concept of Nash optimally $($recall Def. $\ref{def3ch5})$.\\
It is clear that $(\ref{eq3ch5})$-$(\ref{eq4ch5})$ is equivalent to
\begin{equation}
\label{eq14ch5}
\left\{
\begin{array}{rl}
&\displaystyle \alpha_i \iint_{\mathcal{O}_{i,d} \times (0,T)} (y-y_{i,d}) \hat{y}^i\ dxdt + \mu_i \iint_{\mathcal{O}_i \times (0,T)} v^i \hat{v}^i\ dxdt = 0,\\
&\forall \hat{v}^i \in L^2(\mathcal{O}_i \times (0,T)), \ \  {v}^i \in L^2(\mathcal{O}_i \times (0,T)),\ \ i=1,2,
\end{array}
\right.
\end{equation}
where we have denoted by $\hat{y}^i$ the derivative of the state $y$ with respect to $v^i$ in the direction $\hat{v}^i$. One has
\begin{equation}
\label{EC3ch5}
\left\{
\begin{array}{rl}
&\hat{y}^i_t -  ((D_1 a(y_x,t,x)y_x + a(y_x,t,x))\hat{y^i}_x)_{x} + D_1 F(y,y_x) \hat{y}^i + D_2 F(y,y_x) \hat{y}^i_x = \hat{v}^i 1_{\mathcal{O}} \ \text{in} \ Q,\\
&\hat{y}^i(0,t) =\hat{y}^i(L,t) = 0\ \ \text{ in } \  (0,T),\\
&\hat{y}^i(0) = 0\ \ \text{ in } \ I.
\end{array}
\right.
\end{equation}
Let us introduce the adjoint systems for $(\ref{EC3ch5})$
\begin{equation}
\label{EC4ch5}
\left\{
\begin{array}{rl}
&-p^i_t - ((D_1 a(y_x,t,x)y_x + a(y_x,t,x)) p^i_x)_{x} + D_1 F(y,y_x) p^i - (D_2 F(y,y_x) p^i)_x = \alpha_i (y-y_{i,d}) {1}_{\mathcal{O}_{i,d}} \ \text{in} \ Q,\\
&p^i(0,t) = p^i(L,t) = 0\ \ \text{ in } \ (0,T),\\
&p^i(T) = 0\ \ \text{in } \ I.
\end{array}
\right.
\end{equation}
If we multiply $(\ref{EC4ch5})_1$ by $\hat{y}^i$ in $L^2(Q)$, and perform integration by parts, we obtain
\begin{equation*}
\alpha_i \iint_Q (y-y_{i,d}) 1_{\mathcal{O}_{i,d}} \hat{y}^i\ dxdt = \iint_Q p^i \hat{v}^i\ 1_{\mathcal{O}_i}\ dxdt.
\end{equation*}
Replacing the above expression in $(\ref{eq14ch5})$, we have
\begin{equation*}
\iint_Q p^i \hat{v}^i 1_{\mathcal{O}_i}\ dxdt + \mu_i \iint_{{\mathcal{O}_i} \times (0,T)} v^i \hat{v}^i\ dxdt = 0.
\end{equation*}
As a consequence, we get the following characterization of any Nash quasi-equilibrium for $J_i$
\begin{equation}
\label{eq15}
v^i = - \dfrac{1}{\mu_i} p^i 1_{\mathcal{O}_i}.
\end{equation}
In this way, we have the following optimality system for $(\ref{EC1ch5})$
\begin{equation}
\label{EC7ch5}
\left\{
\begin{array}{rl}
&y_t -  (a(y_x,t,x) y_x)_x + F(y,y_x) = f 1_\mathcal{O} - \dfrac{1}{\mu_1} p^1 1_{\mathcal{O}_1} - \dfrac{1}{\mu_2} p^2 1_{\mathcal{O}_2} \ \text{in} \ Q,\\
&-p^i_t - ((D_1 a(y_x,t,x)y_x + a(y_x,t,x)) p^i_x)_{x} + D_1 F(y,y_x) p^i - (D_2 F(y,y_x) p^i)_x = \alpha_i (y-y_{i,d}) 1_{\mathcal{O}_{i,d}} \ \text{in} \ Q,\\
&y(0,t) = y(L,t) = 0, \ \  p^i(0,t)=p^i(L,t)=0\ \ \text{ in } \ (0,T),\\
&y(0) = y_{0}, \ \  p^i(T)=0\ \ \text{in } \ I.
\end{array}
\right.
\end{equation}
We consider the linearized system for $(\ref{EC7ch5})$
\begin{equation}
\label{eq18ch5}
\left\{
\begin{array}{rl}
&y_t - (a(0,t,x) y_{x})_x + D_1 F(0,0)y + D_2F(0,0)y_x =  f 1_\mathcal{O} - \dfrac{1}{\mu_1} p^1 1_{\mathcal{O}_1} - \dfrac{1}{\mu_2} p^2 1_{\mathcal{O}_2} + G\ \text{in} \ Q,\\
&-p^i_t - (a(0,t,x) p^i_{x})_x + D_1 F(0,0) p^i - D_2F(0,0) p^i_x = \alpha_i y 1_{\mathcal{O}_{i,d}} + G_i\ \text{in} \ Q,\\
&y(0,t) = y(L,t) = 0, \ \  p^i(0,t)=p^i(L,t)=0\ \ \text{ in } \ (0,T),\\
&y(0) = y_{0},\ \ p^i(T)=0, \ \ \text{ in } \ I.
\end{array}
\right.
\end{equation}
Now, we consider the adjoint system for $(\ref{eq18ch5})$
\begin{equation}
\label{eq20ch5}
\left\{
\begin{array}{rl}
&-\varphi_t -  (a(0,t,x) \varphi_{x})_x + D_1 F(0,0) \varphi - D_2F(0,0) \varphi_x=  \alpha_1 \theta^1 1_{\mathcal{O}_{1,d}} + \alpha_2 \theta^2 1_{\mathcal{O}_{2,d}} + \mathcal{G} \ \text{in} \ Q,\\
&\gamma^i_t -  (a(0,t,x) \gamma^i_{x})_x + D_1 F(0,0) \gamma^i + D_2 F(0,0) \gamma^i_x =  - \dfrac{1}{\mu_i} \varphi 1_{\mathcal{O}_i} + \mathcal{G}_i\ \text{in} \ Q,\\
&\varphi(0,t) = \varphi(L,t) = 0, \ \ \gamma^i(0,t)=\gamma^i(L,t) =0, \ \ \text{ in } \ (0,T),\\
&\varphi(T) = \varphi^{T},\ \ \gamma^i(0)=0, \ \ \text{ in } \ I.
\end{array}
\right.
\end{equation}


\section{Null Controllability for Linearized System (\ref{eq18ch5})}\label{s4}

Note that to prove the existence and uniqueness of a Nash Quasi-Equilibrium for the linearized system of (\ref{EC1ch5}) is equivalent to prove the null controllability to the linear optimality system.\\
For this purpose, we apply the Carleman techniques in the adjoint system (\ref{eq20ch5}), in this way, we will need to define weight functions.\\
Let us consider a non-empty open set $\tilde{\mathcal{O}} \subset \subset \mathcal{O}$ such that $\mathcal{O}_{i,d} \cap \tilde{\mathcal{O}} \neq \emptyset$ for $i=1,2$.\\
If $(\ref{condition1ch5})$ is satisfied, we can define $\mathcal{O}_d := \mathcal{O}_{1,d} = \mathcal{O}_{2,d}$ and we introduce the non-empty open set $\omega_0$ satisfying $\omega_0 \subset \subset \mathcal{O}_{d} \cap \tilde{\mathcal{O}}$.
\begin{lemma}[Fursikov's Lemma]
\label{l1prelim}
There exists a function $\eta_0 = \eta_0(x) \in C^2(\overline{I})$ satisfying 
\begin{equation*}
\left\{
\begin{array}{rl}
&\eta_0 >0,\ \text{in}\ I,\ \ \eta_i =0\ \text{on}\ \partial I, \\
&|\eta_{0,x}|> 0\ \text{in}\ \ \overline{I}
 \backslash \omega_0. 
\end{array}
\right.
\end{equation*}
\end{lemma}
\begin{proof}
See \cite{Furs}.
\end{proof}
If $(\ref{condition2ch5})$ is satisfied, we introduce the non-empty connected open sets $\omega_i$ with
\begin{equation}
\label{condition3ch5}
\omega_i \subset \subset \mathcal{O}_{i,d} \cap \tilde{\mathcal{O}},\ i=1,2\ \ \omega_1 \cap \omega_2 \neq \emptyset.
\end{equation}
such that
\begin{lemma}
\label{lemfursikov2ch5} 
There exists functions $\eta_i=\eta_i(x) \in C^2(\overline{I})$ $(i=1,2)$ satisfying
\begin{equation*}
\left\{
\begin{array}{rl}
&\eta_i >0,\ \text{in}\ I,\ \ \eta_i =0\ \text{on}\ \partial I, \\
&|\eta_{i,x}|> 0\ \text{in}\ \ \overline{I}
 \backslash \omega_i,\ \ \eta_1=\eta_2 \ \text{in}\ I\backslash \tilde{\mathcal{O}}. 
\end{array}
\right.
\end{equation*}
\end{lemma}
\begin{proof}
See \cite{Cara5}.
\end{proof}
\begin{observation}
Lemma $\ref{lemfursikov2ch5}$ establishes the existence of functions $\eta_1$ and $\eta_2$ which coincide outside $\tilde{\mathcal{O}}$ but may be very different inside $\tilde{\mathcal{O}}$. Nevertheless, it will be seen in the proof that one can find $\eta_1$ and $\eta_2$ satisfying $\|\eta_1\|_\infty=\|\eta_2\|_\infty$.  
\end{observation}
\begin{observation}
From $(\ref{condition1ch5})$, $(\ref{eq8ch5})$ and $(\ref{condition3ch5})$, we see that it can be assumed that either
\begin{equation}
\label{condition4ch5}
\omega_1 \cap \mathcal{O}_{2,d} = \emptyset \ \ \text{and}\ \ \ \omega_2 \cap \mathcal{O}_{1,d} = \emptyset
\end{equation}
or
\begin{equation}
\label{condition5ch5}
\omega_i \subset \mathcal{O}_{j,d}\ \ \text{and}\ \ \omega_j \cap \mathcal{O}_{i,d} = \emptyset,\ \ \text{with}\ \ (i,j)=(1,2)\ \text{or}\ (i,j)=(2,1). 
\end{equation}
\end{observation}

Let us introduce the weight functions
$$\sigma(x,t):= \dfrac{e^{4\lambda\|\eta_0\|_{L^\infty(\Omega)}} - e^{\lambda (2\|\eta_0\|_{L^\infty(\Omega)}+ \eta_0(x))}}{t(T-t)}, \ \ \xi(x,t):= \dfrac{e^{\lambda (2\|\eta_0\|_{L^\infty(\Omega)}+ \eta^0(x))}}{t(T-t)},$$
$$\sigma_i(x,t):= \dfrac{e^{4\lambda\|\eta_i\|_{L^\infty(\Omega)}} - e^{\lambda (2\|\eta_i\|_{L^\infty(\Omega)}+ \eta_i(x))}}{t(T-t)}, \ \ \xi_i(x,t):= \dfrac{e^{\lambda (2\|\eta_i\|_{L^\infty(\Omega)}+ \eta_i(x))}}{t(T-t)},$$
and the notations
\begin{equation}
\label{notation1ch5}
I_m(\psi):= s^{m-4} \lambda^{m-3} \iint_Q e^{-2s \sigma} (\xi)^{m-4}(|\psi_t|^2 + |\psi_{xx}|^2)dxdt + L_m (\psi),
\end{equation}
\begin{equation}
\label{notation2ch5}
I^i_m(\psi):= s^{m-4} \lambda^{m-3} \iint_Q e^{-2s \sigma_i} (\xi_i)^{m-4}(|\psi_t|^2 + |\psi_{xx}|^2)dxdt + L^i_m (\psi),
\end{equation}
where
$$L_m(\psi):= s^{m-2} \lambda^{m-1} \iint_Q e^{-2s \sigma} (\xi)^{m-2} |\psi_x|^2 dxdt + s^{m} \lambda^{m+1} \iint_Q e^{-2s \sigma} (\xi)^m |\psi|^2 dxdt,$$
$$L^i_m(\psi):= s^{m-2} \lambda^{m-1} \iint_Q e^{-2s \sigma_i} (\xi_i)^{m-2} |\psi_x|^2 dxdt + s^{m} \lambda^{m+1} \iint_Q e^{-2s \sigma_i} (\xi_i)^m |\psi|^2 dxdt.$$

\begin{propposition}
\label{prop1ch5}
Assume that $(\ref{condition1ch5})-(\ref{eq8ch5})$ are satisfied. Then, there exists $C(I,\mathcal{O})>0$ such that, for every $s\geq C(T+T^2)$ and every $\lambda \geq C$, the solution $(\varphi,\gamma^1,\gamma^2)$ to $(\ref{eq20ch5})$ associated to $\varphi^T \in L^2(I)$ satisfies the following:
\begin{itemize}
\item[i)] If $(\ref{condition2ch5})$ holds, then
\begin{align}
\label{carleman1ch5}
I_0(\varphi) + I_0(h) \leq &C \left(s^{-3} \lambda^{-2}\iint_Q e^{-2s \sigma} (|\mathcal{G}|^2 + |\mathcal{G}_1|^2 + |\mathcal{G}_2|^2) dxdt \right. \nonumber \\ 
&+ \left. s^4 \lambda^5 \iint_{\mathcal{O} \times (0,T)} e^{-2s \sigma} \xi^4 |\varphi|^2dxdt \right).
\end{align}
\item[ii)] If $(\ref{condition4ch5})$ holds, then
\begin{align}
I^1_0(\gamma^1) + I^2_0(\gamma^2) + &s^{-3} \lambda^{-2} \iint_Q e^{-2s \sigma_1} (\xi_1)^{-3} |\varphi|^2 dxdt \nonumber\\
\leq &C \left( s^{-3} \lambda^{-3} \iint_Q (e^{-2s \sigma_1} + e^{-2s \sigma_2}) (|\mathcal{G}|^2 + |\mathcal{G}_1|^2 + |\mathcal{G}_2|^2) dxdt \right. \nonumber\\
\label{carleman2ch5}
&+ \left. s^4 \lambda^5 \iint_{\mathcal{O} \times (0,T)} (e^{-2s \sigma_1} (\xi_1)^4 + e^{-2s \sigma_2} (\xi_2)^4)|\varphi|^2 dxdt \right).
\end{align}

\item[iii)] If $(\ref{condition5ch5})$ holds for $(i,j)=(i_0,j_0)$ with $(i_0,j_0)=(1,2)$ or $(i_0,j_0)=(2,1)$ ,then

\begin{align}
I^{j_0}_0 (\gamma^{j_0}) + I^{i_0}_0 (h) + &s^{-3} \lambda^{-2} \iint_Q e^{-2s \sigma_{j_0}} (\xi_{j_0})^{-3} |\varphi|^2 dxdt  \nonumber \\
\leq &C \left(s^{-3} \lambda^{-3}\iint_Q (e^{-2s \sigma_1} + e^{-2s \sigma_2}) (|\mathcal{G}|^2 + |\mathcal{G}_1|^2 + |\mathcal{G}_2|^2) dxdt \right. \nonumber\\
\label{carleman3ch5}
&+ \left. s^4 \lambda^5 \iint_{\mathcal{O} \times (0,T)} (e^{-2s \sigma_1} (\xi_1)^4 + e^{-2s \sigma_2} (\xi_2)^4)|\varphi|^2dxdt \right).
\end{align}

\end{itemize}
where $h:= \alpha_1 \gamma^1 + \alpha_2 \gamma^2$.
\end{propposition}

\begin{proof}
For $i)$ see \cite{Cara1}, and for $ii)$ and $iii)$ see \cite{Cara5}.
\end{proof}

We will apply a standard observability argument, in fact let us consider the following weight functions
\begin{equation*}
l(t):= 
\begin{cases}
&T^2/4, \ \text{for}\ \ 0 \leq t \leq T/2,\\
&t(T-t), \ \text{for}\ \ T/2 \leq t \leq T,
\end{cases}
\end{equation*}
and
$$\overline{\sigma}(x,t):= \dfrac{e^{4\lambda\|\eta^0\|_{L^\infty(\Omega)}} - e^{\lambda (2\|\eta^0\|_{L^\infty(\Omega)}+ \eta^0(x))}}{l(t)}, \ \ \overline{\xi}(x,t):= \dfrac{e^{\lambda (2\|\eta^0\|_{L^\infty(\Omega)}+ \eta^0(x))}}{l(t)},$$
$$\overline{\sigma}_i(x,t):= \dfrac{e^{4\lambda\|\eta_i\|_{L^\infty(\Omega)}} - e^{\lambda (2\|\eta_i\|_{L^\infty(\Omega)}+ \eta_i(x))}}{l(t)}, \ \ \overline{\xi}_i(x,t):= \dfrac{e^{\lambda (2\|\eta_i\|_{L^\infty(\Omega)}+ \eta_i(x))}}{l(t)}.$$
We consider
$$\sigma^*(t):= \underset{x \in \Omega}{\text{max}}\ \overline{\sigma}(x,t),\ \ 
\hat{\sigma}(t):= \underset{x \in \Omega}{\text{min}}\ \overline{\sigma}(x,t),\ \ \xi^*(t):=\underset{x \in \Omega}{\text{max}}\ \overline{\xi}(x,t),$$
$$\sigma_i^*(t):= \underset{x \in \Omega}{\text{max}}\ \overline{\sigma}_i(x,t),\ \ 
\hat{\sigma}_i(t):= \underset{x \in \Omega}{\text{min}}\ \overline{\sigma}_i(x,t), \ \ \xi^*_i(t):=\underset{x \in \Omega}{\text{max}}\ \overline{\xi}_i(x,t).$$
If $\lambda > 1 / \|\eta^0\|_{\infty}$ and $\lambda > 1 / \|\eta_i\|_{\infty}$ (sufficiently large), we have
\begin{equation}
\label{weight-estimate1ch5}
\hat{\sigma} \leq \overline{\sigma} < \frac{5}{4} \hat{\sigma}, \  \ \frac{4}{5}\sigma^*< \overline{\sigma} \leq \sigma^*,
\end{equation}
\begin{equation}
\label{weight-estimate2ch5}
\hat{\sigma}_i \leq \overline{\sigma}_i < \frac{5}{4} \hat{\sigma}_i, \  \ \frac{4}{5}\sigma^*_i< \overline{\sigma}_i \leq \sigma^*_i.
\end{equation}

Let us denote by $\overline{I}_m(\varphi)$ the right-hand side of $(\ref{notation1ch5})$ with $\sigma$ and $\xi$ respectively replaced by $\overline{\sigma}$ and $\overline{\xi}$. Then, one can directly see from the energy estimate and the Proposition $\ref{prop1ch5}$ that
\begin{itemize}
\item[i)] If $(\ref{condition2ch5})$ holds, then 
\begin{align*}
\|\varphi(0)\|^2+\overline{I}_0(\varphi) + \overline{I}_0(h) \leq &C \left(s^{-3} \lambda^{-2}\iint_Q e^{-2s \overline{\sigma}} (|\mathcal{G}|^2 + |\mathcal{G}_1|^2 + |\mathcal{G}_2|^2) dxdt \right. \nonumber \\ 
&+ \left. s^4 \lambda^5 \iint_{\mathcal{O} \times (0,T)} e^{-2s \overline{\sigma}} \overline{\xi}^4 |\varphi|^2dxdt \right).
\end{align*}
\item[ii)] If $(\ref{condition4ch5})$ holds, then
\begin{align*}
\|\varphi(0)\|^2 + &s^{-3} \lambda^{-2} \iint_Q e^{-2s \overline{\sigma}_1} (\overline{\xi}_1)^{-3} |\varphi|^2 dxdt \nonumber\\
\leq &C \left( s^{-3} \lambda^{-3} \iint_Q (e^{-2s \overline{\sigma}_1} + e^{-2s \overline{\sigma}_2}) (|\mathcal{G}|^2 + |\mathcal{G}_1|^2 + |\mathcal{G}_2|^2) dxdt \right. \nonumber\\
&+ \left. s^4 \lambda^5 \iint_{\mathcal{O} \times (0,T)} (e^{-2s \overline{\sigma}_1} (\overline{\xi}_1)^4 + e^{-2s \overline{\sigma}_2} (\overline{\xi}_2)^4)|\varphi|^2 dxdt \right).
\end{align*}
\item[iii)] If $(\ref{condition5ch5})$ holds for $(i,j)=(i_0,j_0)$ with $(i_0,j_0)=(1,2)$ or $(i_0,j_0)=(2,1)$ ,then
\begin{align*}
\|\varphi(0)\|^2 + &s^{-3} \lambda^{-2} \iint_Q e^{-2s \overline{\sigma}_{j_0}} (\tilde{\xi}_{j_0})^{-3} |\varphi|^2 dxdt  \nonumber \\
\leq &C \left(s^{-3} \lambda^{-3}\iint_Q (e^{-2s \overline{\sigma}_1} + e^{-2s \overline{\sigma}_2}) (|\mathcal{G}|^2 + |\mathcal{G}_1|^2 + |\mathcal{G}_2|^2) dxdt \right. \nonumber\\
&+ \left. s^4 \lambda^5 \iint_{\mathcal{O} \times (0,T)} (e^{-2s \overline{\sigma}_1} (\overline{\xi}_1)^4 + e^{-2s \overline{\sigma}_2} (\overline{\xi}_2)^4)|\varphi|^2dxdt \right).
\end{align*}
\end{itemize}
Now, we denote 
\begin{equation}
\label{weight-betach5}
\beta(x,t):= \frac{2}{5} \overline{\sigma}(x,t), \ \ \beta^*(t):= \underset{x \in \Omega}{\text{max}}\ {\beta}(x,t),\ \ \hat{\beta}(t):= \underset{x \in \Omega}{\text{min}}\ {\beta}(x,t)
\end{equation}
and
\begin{equation}
\label{weight-beta_ich5}
\beta_i(x,t):= \frac{2}{5} \overline{\sigma_i}(x,t), \ \ \beta^*_i(t):= \underset{x \in \Omega}{\text{max}}\ {\beta}_i(x,t),\ \ \hat{\beta}_i(t):= \underset{x \in \Omega}{\text{min}}\ {\beta}_i(x,t).
\end{equation}

Using $(\ref{weight-estimate1ch5})-(\ref{weight-beta_ich5})$ in the last result, we get

\begin{itemize}
\item[i)]  If $(\ref{condition2ch5})$ holds, then
\begin{align*}
\|\varphi(0)\|^2 + \iint_Q e^{-5s \beta^*} |\varphi|^2 dxdt \leq &C \left(\iint_Q e^{-4s \beta^*} (|\mathcal{G}|^2 + |\mathcal{G}_1|^2 + |\mathcal{G}_2|^2) dxdt \right. \nonumber \\ 
&+ \left.\iint_{\mathcal{O} \times (0,T)} e^{-4s \beta^*} ({\xi}^*)^4 |\varphi|^2dxdt \right).
\end{align*}
\item[ii)] If $(\ref{condition4ch5})$ holds, then
\begin{align*}
\|\varphi(0)\|^2 +  \iint_Q e^{-5 s \beta^*_1} ({\xi}^*_1)^{-3} |\varphi|^2 dxdt \leq &C \left( \iint_Q e^{-4s \beta^*_1} (|\mathcal{G}|^2 + |\mathcal{G}_1|^2 + |\mathcal{G}_2|^2) dxdt \right. \\
&+ \left. \iint_{\mathcal{O} \times (0,T)} e^{-4s \beta^*_1} ({\xi}^*_1)^4 |\varphi|^2 dxdt \right).
\end{align*}
\item[iii)] If $(\ref{condition5ch5})$ holds for $(i,j)=(i_0,j_0)$ with $(i_0,j_0)=(1,2)$ or $(i_0,j_0)=(2,1)$ ,then
\begin{align*}
\|\varphi(0)\|^2 + \iint_Q e^{-5s \beta^*_{j_0}} ({\xi}^*_{j_0})^{-3} |\varphi|^2 dxdt \leq &C \left( \iint_Q e^{-4s \beta^*_1} (|\mathcal{G}|^2 + |\mathcal{G}_1|^2 + |\mathcal{G}_2|^2) dxdt \right. \\
&+ \left. \iint_{\mathcal{O} \times (0,T)} e^{-4s \beta^*_1} ({\xi}^*_1)^4 |\varphi|^2 dxdt \right).
\end{align*}
\end{itemize}

Taking the PDE satisfied by the $\gamma^i$ in $(\ref{eq20ch5})$, multiplying by $e^{-5s \beta} \gamma^i$ or $e^{-5s \beta_j} (\xi^*_j)^{-3} \gamma^i$, we easily see that
\begin{align*}
\iint_Q e^{-5s \beta^*} |\gamma^i|^2 dxdt \leq C \left( \iint_Q e^{-5s \beta^*} (|\mathcal{G}|^2 + |\mathcal{G}_1|^2 + |\mathcal{G}_2|^2) dxdt + \iint_Q e^{-5s \beta^*} |\varphi|^2 dxdt \right)
\end{align*}
or
\begin{align*}
\iint_Q e^{-5s \beta^*_j} (\xi^*_j)^{-3} |\gamma^i|^2 dxdt \leq &C \left( \iint_Q e^{-5s \beta^*_j} (\xi^*_j)^{-3} (|\mathcal{G}|^2 + |\mathcal{G}_1|^2 + |\mathcal{G}_2|^2) dxdt  \right.\\
&+ \left. \iint_Q e^{-5s \beta^*_j} (\xi^*_j)^{-3} |\varphi|^2 dxdt \right)
\end{align*}
Then, joined the last results we obtain
\begin{align*}
\|\varphi(0)\|^2 + \iint_Q e^{-5s \beta^*} |\varphi|^2 dxdt + &\iint_Q e^{-5s \beta^*} (|\gamma^1|^2 + |\gamma^2|^2) dxdt\\
\leq &C \left(\iint_Q e^{-4s \beta^*} (|\mathcal{G}|^2 + |\mathcal{G}_1|^2 + |\mathcal{G}_2|^2) dxdt \right. \nonumber \\ 
&+ \left.\iint_{\mathcal{O} \times (0,T)} e^{-4s \beta^*} ({\xi}^*)^4 |\varphi|^2dxdt \right).
\end{align*}
or
\begin{align*}
\|\varphi(0)\|^2 + \iint_Q e^{-5s \beta_1^*} (\xi^*_1)^{-3} |\varphi|^2 dxdt + &\iint_Q e^{-5s \beta^*_1} (\xi^*_1)^{-3} (|\gamma^1|^2 + |\gamma^2|^2) dxdt\\
\leq &C \left(\iint_Q e^{-4s \beta^*_1} (|\mathcal{G}|^2 + |\mathcal{G}_1|^2 + |\mathcal{G}_2|^2) dxdt \right. \nonumber \\ 
&+ \left.\iint_{\mathcal{O} \times (0,T)} e^{-4s \beta^*_1} ({\xi}^*_1)^4 |\varphi|^2dxdt \right).
\end{align*}

Finally, for the two cases, we have the new observability inequality

\begin{align}
\label{observabilitych5}
\|\varphi(0)\|^2 + &\iint_Q e^{-5s \overline{\beta}^*} (\overline{\xi}^*)^{-3} (|\varphi|^2 +|\gamma^1|^2 + |\gamma^2|^2) dxdt \nonumber\\
\leq &C \left(\iint_Q e^{-4s \overline{\beta}^*} (|\mathcal{G}|^2 + |\mathcal{G}_1|^2 + |\mathcal{G}_2|^2) dxdt \right. \nonumber \\ 
&+ \left.\iint_{\mathcal{O} \times (0,T)} e^{-4s \overline{\beta}^*} (\overline{\xi}^*)^4 |\varphi|^2dxdt \right).
\end{align}

where $(\overline{\beta}^*,\overline{\xi}^*):= (\beta^*,\xi^*)$ or $(\beta^*_1,\xi^*_1)$.

Let us define
\begin{equation}
\label{eq29ch5}
\begin{aligned}
&\rho:= e^{5s \overline{\beta}^*/2} (\overline{\xi}^*)^{3/2},\ \ \rho_0 := e^{2s\overline{\beta}^*}, \ \ \rho_1 := e^{2s\overline{\beta}^*} (\overline{\xi}^*)^{-2}, \\ 
&{\rho}_2:= e^{3s\overline{\beta}^*/2} (\overline{\xi}^*)^{-3},\ \ \rho_3 := e^{3s\overline{\beta}^*/2} (\overline{\xi}^*)^{-8},  \ \ \rho_4 = e^{3s\overline{\beta}^*/2} (\overline{\xi}^*)^{-9}, \ \ {\rho}_5 := e^{3s\overline{\beta}^*/2} (\overline{\xi}^*)^{-10}. 
\end{aligned}
\end{equation}
\begin{propposition}
\label{prop3ch5}
Assume that $ \rho G \in L^2(Q)$, ${\rho}_3 G_{t}\in L^2(Q)$, $\rho G_i \in L^2(Q)$ and $G(0) \in H_0^1(I)$ $(i=1,2)$. Then $(\ref{eq18ch5})$ is null-controllable. More precisely, for any $y_0 \in H^3(I) \cap H^1_0(I)$, there exists a control-state $(y,p^1,p^2,f)$ satisfying
\begin{equation}
\label{eq31'ch5}
f \in L^2(\mathcal{O} \times (0,T)),\ \ y,p^1,p^2 \in L^\infty(0,T;H_0^1(I)) \cap L^2(0,T;H^2(I))
\end{equation}
such that
\begin{equation}
\label{eq31ch5}
\iint_Q \rho_0^2 (|y|^2+|p^1|^2+|p^2|^2)dxdt<+\infty,\ \ \iint_{\mathcal{O} \times (0,T)} (\rho_1^2 |f|^2 + {\rho_3}^2 |f_t|^2) dxdt <+\infty.
\end{equation}
In particular $y(T)=0$.
\end{propposition}
\begin{proof}
Let us denote $Lw=w_t - (a(0,t,x)w_{x})_x + D_1F(0,0) w + D_2F(0,0) w_x\ $ and $\ L^*w=-w_t - (a(0,t,x)w_{x})_x + D_1F(0,0) w - D_2F(0,0) w_x$, then, we define a vectorial space
$$\mathcal{X}_0 := \{(u,z^1,z^2) \in C^2(\overline{I})^3; u = 0,\  z^1=z^2=0\ \ \text{on}\ \Sigma,\ \ z^1(0)=z^2(0)=0 \}.$$
and an application   $\ \ b: \mathcal{X}_0 \times \mathcal{X}_0 \rightarrow \mathbb{R}$
\begin{align*}
b((u,&z^1,z^2),(\tilde{u},\tilde{z},\tilde{z}^1,\tilde{z}^2))\\
:= &\iint_Q \rho_0^{-2} (L^*u - \alpha_1 z^1 {1}_{\mathcal{O}_{1,d}} - \alpha_2 z^2 {1}_{\mathcal{O}_{2,d}})(L^*\tilde{u} - \alpha_1 \tilde{z}^1 {1}_{\mathcal{O}_{1,d}} - \alpha_2 \tilde{z}^2 {1}_{\mathcal{O}_{2,d}})dxdt\\
&+ \sum_{i=1}^2 \iint_Q \rho_0^{-2} (Lz^i  + \dfrac{1}{\mu_i} u {1}_{\mathcal{O}_i})(L\tilde{z}^i + \dfrac{1}{\mu_i} \tilde{u} {1}_{\mathcal{O}_i})dxdt\\
&+ \iint_{\mathcal{O} \times (0,T)} \rho_1^{-2}\ u \tilde{u}\ dxdt,\ \ \ \ \forall (u,z^1,z^2),(\tilde{u},\tilde{z}^1,\tilde{z}^2) \in \mathcal{X}_0.
\end{align*}
We will prove that $b(\cdot,\cdot)$ defines an inner product, for that, it is enough to prove:\\ 
If $b((u,z^1,z^2),(u,z^1,z^2))=0$,\ then $(u,z^1,z^2)=(0,0,0)$. Indeed, we have
\begin{align*}
\iint_Q &\rho_0^{-2} |L^*u - \alpha_1 z^1 {1}_{\mathcal{O}_{1,d}} - \alpha_2 z^2 {1}_{\mathcal{O}_{2,d}}|^2 dxdt \\
&+ \sum_{i=1}^2 \iint_Q \rho_0^{-2} |Lz^i + \dfrac{1}{\mu_i} u {1}_{\mathcal{O}_i}|^2 dxdt + \iint_{\mathcal{O} \times (0,T)} \rho_1^{-2} |u|^2 dxdt =0.
\end{align*}
Thus, we obtain the system
\begin{equation}
\label{eq32ch5}
\left|
\begin{array}{rl}
&L^*u = 0 +  \alpha_1 z^1 {1}_{\mathcal{O}_{1,d}} +  \alpha_2 z^2 {1}_{\mathcal{O}_{2,d}} \ \text{in} \ Q,\\
&Lz^i = 0 - \dfrac{1}{\mu_i} u {1}_{\mathcal{O}_i} \ \text{in} \ Q,\\
&u(0,t)=u(L,t) = 0, \ \ z^i(0,t)=z^i(L,t) =0, \ \ \text{ in } \ (0,T),\\
&u(T) = u^{T},\ \ z^i(0)=0, \ \ \text{ in } \ I.
\end{array}
\right.
\end{equation}
For the Proposition $\ref{prop1ch5}$ on the system $(\ref{eq32ch5})$, we have
\begin{align*}
\|u(0)\|^2 + &\iint_Q e^{-5s \overline{\beta}^*} (\overline{\xi}^*)^{-3} (|u|^2 +|z^1|^2 + |z^2|) dxdt \\
\leq &C \left(\iint_Q \rho_0^{-2} (|L^*u - \alpha_1 z^1 {1}_{\mathcal{O}_{1,d}} - \alpha_2 z^2 {1}_{\mathcal{O}_{2,d}}|^2 \right.\\
&+ |Lz^1 + \dfrac{1}{\mu_1} u {1}_{\mathcal{O}_1}|^2 + |Lz^2 + \dfrac{1}{\mu_2} u {1}_{\mathcal{O}_2}|^2) dxdt   \\ 
&+ \left.\iint_{\mathcal{O} \times (0,T)} \rho_1^{-2} |\varphi|^2dxdt \right)=0.
\end{align*}
Then $(u,z^1,z^2)=(0,0,0)$. This proves that $b(\cdot,\cdot)$ define a inner product in $\mathcal{X}_0$.\\ 
Now, let us define $\mathcal{X}$ the completion of $\mathcal{X}_0$ with this inner product, then $\mathcal{X}$ is a Banach space with norm induced by the inner product $b(\cdot,\cdot)$. Clearly $b(\cdot,\cdot)$ is a bilinear, symmetric, continuous and coercive application in $\mathcal{X}$.\\
Let us define the functional linear\ $\mathbb{G}: \mathcal{X} \rightarrow \mathbb{R}\ $ as
\begin{equation*}
<\mathbb{G},(u,z^1,z^2)>:= (y_0,u(0))+ \iint_Q (G u + {G}_1 z^1 + {G}_{2} z^2 )dxdt.
\end{equation*}
Let us see that $\mathbb{G}$ is continuous. Indeed,  if $(u,z^1,z^2) \in \mathcal{X}$, we have
\begin{align*}
|<\mathbb{G},(u,z^1,z^2)>| \leq  &|(y_0,u(0))| + \iint_Q (||G|u||{G}_1| |z^1|  + |{G}_{2}| |z^2|)dxdt\\
\leq &\|y_0\| \|u(0)\| + \left\{\iint_Q \rho^2  (|G|^2 + |{G}_1|^2 + |{G}_{2}|^2) dxdt \right\}^{\frac{1}{2}}\cdot\\
&\left\{ \iint_Q \rho^{-2} (|u|^2 + |z^1|^2 + |z^2|^2) dxdt  \right\}^{\frac{1}{2}} \\
\leq &\left\{\|y_0\|^2 + \iint_Q \rho^2 (|G|^2 + |{G}_1|^2 + |{G}_{2}|^2) dxdt \right\}^{\frac{1}{2}} \\
&\left\{ \|u(0)\|^2 + \iint_Q \rho^{-2} (|u|^2 + |z^1|^2 + |z^2|^2) dxdt\right\}^{\frac{1}{2}}\\
\leq &C\ b((u,z^1,z^2),(u,z^1,z^2))^{\frac{1}{2}} = C\ \|(u,z^1,z^2)\|_{_{\mathcal{X}}}.
\end{align*}
Then, for Lax-Milgram's Theorem, there exists an unique $(\hat{u},\hat{z}^1,\hat{z}^2)\in \mathcal{X}$ such that
\begin{equation}
b((\hat{u},\hat{z}^1,\hat{z}^2),(u,z^1,z^2))=<\mathbb{G},(u,z^1,z^2)>,\ \ \forall (u,z^1,z^2) \in \mathcal{X}.
\end{equation}
In other words,
\begin{align}
\label{eq33ch5}
\nonumber
\iint_Q &\rho_0^{-2} (L^*\hat{u} - \alpha_1 \hat{z}^1 {1}_{\mathcal{O}_{1,d}} - \alpha_2 \hat{z}^2 {1}_{\mathcal{O}_{2,d}})(L^*u - \alpha_1 z^1 {1}_{\mathcal{O}_{1,d}} - \alpha_2 z^2 {1}_{\mathcal{O}_{2,d}})dxdt\\
\nonumber
+ & \sum_{i=1}^2 \iint_Q \rho_0^{-2} (L\hat{z}^i + \dfrac{1}{\mu_i} \hat{u} {1}_{\mathcal{O}_i})(Lz^i + \dfrac{1}{\mu_i} u {1}_{\mathcal{O}_i})dxdt\\
\nonumber
+ &\iint_{\mathcal{O} \times (0,T)} \rho_1^{-2}\ \hat{u} u\ dxdt\\
=& (y_0,u(0))+ \iint_Q ({G} u + G_1 z^1 + {G}_{2} z^2 )dxdt.
\end{align}
As $(\hat{u},\hat{z}^1,\hat{z}^2) \in \mathcal{X}$, then
\begin{equation*}
\begin{cases}
&\rho_0^{-1} (L^*\hat{u} - \alpha_1 \hat{z}^1 {1}_{\mathcal{O}_{1,d}} - \alpha_2 \hat{z}^2 {1}_{\mathcal{O}_{2,d}})\in L^2(Q),\\
&\rho_0^{-1} (L\hat{z}^i + \dfrac{1}{\mu_i} \hat{u} {1}_{\mathcal{O}_i}) \in L^2(Q),\\
&\rho_1^{-1} \hat{u} \in L^2(Q). 
\end{cases}
\end{equation*}
We define
\begin{equation}
\label{eq34ch5}
\begin{cases}
\hat{y}:=\rho_0^{-2} (L^*\hat{u} - \alpha_1 \hat{z}^1 {1}_{\mathcal{O}_{1,d}} - \alpha_2 \hat{z}^2 {1}_{\mathcal{O}_{2,d}})\ \ &\text{in}\ \ Q, \\
\hat{p}^i := \rho_0^{-2} (L\hat{z}^i + \dfrac{1}{\mu_i} \hat{u} {1}_{\mathcal{O}_i})\ \ &\text{in}\ \ Q, \\
\hat{f}:=  -\rho_1^{-2}\ \hat{u}\ \ &\text{in}\ \ \mathcal{O} \times (0,T).
\end{cases}
\end{equation}
Replacing $(\ref{eq34ch5})$ in $(\ref{eq33ch5})$, we obtain
\begin{align*}
\iint_Q &\hat{y}(L^*u - \alpha_1 z^1 {1}_{\mathcal{O}_{1,d}} - \alpha_2 z^2 {1}_{\mathcal{O}_{2,d}})dxdt + \sum_{i=1}^2 \iint_Q \hat{p}^i(Lz^i + \dfrac{1}{\mu_i} u {1}_{\mathcal{O}_i})dxdt\\
=& (y_0,u(0))+ \iint_{\mathcal{O} \times (0,T)} \hat{y} u\ dxdt+ \iint_Q (G u {G}_1 z^1 + {G}_{2} z^2)dxdt
\end{align*}
this is,
\begin{align*}
\iint_Q &\hat{y} b dxdt + \sum_{i=1}^2 \iint_Q \hat{p}^i b_i dxdt\\
=& (y_0,u(0))+ \iint_{\mathcal{O} \times (0,T)} \hat{y} u\ dxdt+ \iint_Q (G u + {G}_1 z^1 + {G}_{2} z^2 )dxdt
\end{align*}
where $(u,z^1,z^2)$ is solution of the system
\begin{equation*}
\left|
\begin{array}{rl}
&L^*u = b +  \alpha_1 z^1 {1}_{\mathcal{O}_{1,d}} +  \alpha_2 z^2 {1}_{\mathcal{O}_{2,d}} \ \text{in} \ Q,\\
&Lz^i = b_i - \dfrac{1}{\mu_i} u {1}_{\mathcal{O}_i} \ \text{in} \ Q,\\
&u(0,t) = u(L,t) = 0, \ \ z^i(0,t) = z^i(L,t) =0, \ \ \text{ in } \ (0,T),\\
&u(T) = 0,\ \ z^i(0)=0, \ \ \text{ in } \ I.
\end{array}
\right.
\end{equation*}
Thus, $(\hat{y},\hat{p}^1,\hat{p}^2)$ is a solution by transposition of the problem
\begin{equation}
\left|
\begin{array}{rl}
&L\hat{y} = {G} + \hat{f} {1}_\mathcal{O} - \dfrac{1}{\mu_1} p^1 {1}_{\mathcal{O}_1} - \dfrac{1}{\mu_2} p^2 {1}_{\mathcal{O}_2} \ \text{in} \ Q,\\
&L^* \hat{p}^i = {G}_i +  \alpha_i \hat{y} {1}_{\mathcal{O}_{i,d}} \ \text{in} \ Q,\\
&\hat{y}(0,t)= \hat{y}(L,t) = 0, \ \ \hat{p}^i(0,t)=\hat{p}^i(L,t) =0, \ \ \text{ in } \ (0,T),\\
&\hat{y}(0) = y_{0},\ \ \hat{p}^i(T)=0, \ \ \text{ in } \ I.
\end{array}
\right.
\end{equation}
Since $G, {G}_1, {G}_2$ are regular, using  energy estimates, we have
$$\hat{y},\ \hat{p}^1,\hat{p}^2 \in L^\infty(0,T;H_0^1(I)) \cap L^2(0,T;H^2(I)).$$
Also,
\begin{align*}
\iint_Q \rho_0^2 |\hat{y}|^2 dxdt &= \iint_Q \rho_0^2 \rho_0^{-4} |L^*\hat{u} - \alpha_1 \hat{z}^1 {1}_{\mathcal{O}_{1,d}} - \alpha_2 \hat{z}^2 {1}_{\mathcal{O}_{2,d}}|^2 dxdt \\
&= \iint_Q \rho_0^{-2} |L^*\hat{u} - \alpha_1 \hat{z}^1 {1}_{\mathcal{O}_{1,d}} - \alpha_2 \hat{z}^2 {1}_{\mathcal{O}_{2,d}}|^2 dxdt<+\infty,\\
\iint_Q \rho_0^2 |\hat{p}^i|^2 dxdt &= \iint_Q \rho_0^2 \rho_0^{-4} |L\hat{z}^i + \dfrac{1}{\mu_i} \hat{u} {1}_{\mathcal{O}_i}|^2 dxdt\\
&= \iint_Q \rho_0^{-2} |L\hat{z}^i + \dfrac{1}{\mu_i} \hat{u} {1}_{\mathcal{O}_i}|^2 dxdt<+\infty,\\
\iint_{\mathcal{O}\times(0,T)} \rho_1^2 |\hat{f}|^2 dxdt &= \iint_{\mathcal{O}\times(0,T)} \rho_1^2 \rho_1^{-4} |\hat{u}|^2 dxdt = \iint_{\mathcal{O}\times(0,T)} \rho_1^{-2} |\hat{u}|^2 dxdt<+\infty.
\end{align*}
And from $(\ref{eq34ch5})$, we have
\begin{equation*}
\left|
\begin{array}{rl}
&L^*\hat{w} = H +  \alpha_1 {\hat{h}}^1 {1}_{\mathcal{O}_{1,d}} +  \alpha_2 {\hat{h}}^2 {1}_{\mathcal{O}_{2,d}} \ \text{in} \ Q,\\
&L\hat{h}^i = H_i - \dfrac{1}{\mu_i} \hat{w} {1}_{\mathcal{O}_i} \ \text{in} \ Q,\\
&\hat{w}(0,t)=\hat{w}(L,t) = 0, \ \ {\hat{h}}^i(0,t)=\hat{h}^i(L,t) =0, \ \ \text{ in } \ (0,T),\\
&\hat{w}(T) = 0,\ \ \hat{h}^i(0)=0, \ \ \text{ in } \ I,
\end{array}
\right.
\end{equation*}
where $\hat{w}:= \rho_3 \rho_1^{-2} \hat{u}$, $\hat{h}^i:= \rho_3 \rho_1^{-2} \hat{z}^i$, $H:= \rho_3 \rho_1^{-2} (L^* \hat{u} - \alpha_1 {\hat{z}}^1 {1}_{\mathcal{O}_{1,d}} - \alpha_2 {\hat{z}}^2 {1}_{\mathcal{O}_{2,d}}) + (\rho_3 \rho_1^{-2})_t \hat{u}$ and $H_i:= \rho_3 \rho_1^{-2} (L \hat{z}^i + \frac{1}{\mu_i} \hat{u} 1_{\mathcal{O}_i}) + (\rho_3 \rho_1^{-2})_t \hat{z}^i$.

Using $(\ref{eq29ch5})$, we get
\begin{align*}
\iint_Q |H|^2 dxdt &\leq C \left(\iint_Q \rho_0^{-2} |L^* \hat{u} - \alpha_1 {\hat{z}}^1 {1}_{\mathcal{O}_{1,d}} - \alpha_2 {\hat{z}}^2 {1}_{\mathcal{O}_{2,d}}|^2 dxdt + \iint_Q \rho^{-2} |\hat{u}|^2 dxdt\right)\\
&\leq C b((\hat{u},\hat{z}^1,\hat{z}^2),(\hat{u},\hat{z}^1,\hat{z}^2))
\end{align*}
and
\begin{align*}
\iint_Q |H_i|^2 dxdt &\leq C \left(\iint_Q \rho_0^{-2} |L \hat{z}^i + \frac{1}{\mu_i} \hat{u} 1_{\mathcal{O}_i}|^2dxdt + \iint_Q \rho^{-2} |\hat{z}^i|^2 dxdt\right)\\
&\leq C b((\hat{u},\hat{z}^1,\hat{z}^2),(\hat{u},\hat{z}^1,\hat{z}^2))
\end{align*}
then
\begin{equation}
\label{eq35ch5}
\rho_3 f \in L^2(0,T; H_0^1(I) \cap H^2(I)), (\rho_3 f)_t \in L^2(Q).
\end{equation}
Furthermore, we have
\begin{align}
\label{control-estimatech5}
\|\rho_3 f\|_{L^2(0,T;H_0^1(I))}^2 + \|(\rho_3)_t f\|_{L^2(Q)}^2 \leq& C\ b((\hat{u},\hat{z}^1,\hat{z}^2),(\hat{u},\hat{z}^1,\hat{z}^2)) \nonumber\\
:= &C \left( \iint_Q \rho_0^2 |\hat{y}|^2 dxdt + \sum_{i=1}^2 \iint_Q \rho_0^2 |\hat{p}^i|^2 dxdt \right. \nonumber\\
&+ \left. \iint_{\mathcal{O} \times (0,T)} \rho_1^2 |\hat{f}|^2 dxdt \right) 
\end{align}
Note that, from $(\ref{eq35ch5})$ one has $f_x \in C([0,T-\delta],L^2(I))$.
\end{proof}

\subsection{Additional Estimates}
From $(\ref{eq29ch5})$, we have
\begin{equation}
\label{eq30ch5}
\begin{aligned}
&{\rho}_{i} \leq C \rho_{i-1} \ \ \forall i \in \{\ 1,2,3,4,5\},\\
&{\rho}_0 \leq C \rho \leq C \rho_5^2\\
& |\rho_i \rho_{i,t}| \leq C |\rho_{i-1}|^2, \ \ \forall i \in \{\ 2,3,4,5\}\ .
\end{aligned}
\end{equation}

\begin{propposition}
\label{prop5ch5}
Let the hypotheses in Proposition $\ref{prop3ch5}$ be satisfied and let $f$ and $(y,p^1,p^2)$ satisfy $(\ref{eq31ch5})$. Then one has
\begin{equation}
\label{eq37ch5}
\begin{aligned}
&\underset{[0,T]}{\text{sup}} (\rho_2^2(t) \|y(t)\|^2) + \underset{[0,T]}{\text{sup}} ({\rho}_2^2(t) \|p^i(t)\|^2) + \iint_Q {\rho}_2^2(|y_x|^2 + |p^i_x|^2) dxdt\\
&+\underset{[0,T]}{\text{sup}} (\rho_3^2(t) \|y_x(t)\|^2) + \underset{[0,T]}{\text{sup}} (\rho_3^2(t) \|p_x^i(t)\|^2) + \iint_Q \rho_3^2(|y_{t}|^2 + |y_{xx}|^2 + |p_t^i|^2 + |p_{xx}^i|^2)dxdt \\
\leq &C \left(\|y_0\|^2 + \iint_Q \rho^2 |G|^2 dxdt + \sum_{i=1}^2 \iint_Q \rho^2 |{G}_i|^2 dxdt \right.\\
&+ \left. \iint_{\mathcal{O}\times(0,T)} \rho_1^2 |f|^2 dxdt + \sum_{i=1}^2 \iint_Q \rho_0^2 |p^i|^2 dxdt + \iint_Q \rho_0^2 |y|^2 dxdt \right).
\end{aligned}
\end{equation}
\end{propposition}

\begin{proof}
Multiplying by $\rho^2_2 y$ the equation $(\ref{eq18ch5})_1$ and integrating in $I$, we have
\begin{align*}
\frac{1}{2} \frac{d}{dt} &(\rho_2^2 \|y(t)\|^2) + \frac{a_0}{2} \int_I \rho_2^2 |y_x|^2 dx \\
&\leq C \left(\int_I \rho_2^2 |G|^2 dx + \int_{\mathcal{O}} \rho_2^2 |f|^2 dx + \int_I \rho_2^2 |y|^2 dx + \sum_{i=1}^2 \int_I \rho_2^2 |p^i|^2 dx \right) + \int_I \rho_{2,t} \rho_2 |y|^2 dx\\
&\leq C \left(\int_I \rho^2 |G|^2 dx + \int_{\mathcal{O}} \rho_1^2 |f|^2 dx + \int_I \rho_0^2 |y|^2 dx + \sum_{i=1}^2 \int_I \rho_0^2 |p^i|^2 dx \right).   
\end{align*}
Integrating from $0$ to $t$, we have
\begin{align}
\label{estim-ad1ch5}
\underset{[0,T]}{\text{sup}} (\rho_2^2 \|y(t)\|^2) + \iint_Q \rho_2^2 |y_x|^2 dxdt &\leq C \left(\|y_0\|^2 + \iint_Q \rho^2 |G|^2 dxdt + \iint_{\mathcal{O} \times (0,T)} \rho_1^2 |f|^2 dxdt \right. \nonumber\\
&+ \left. \iint_Q \rho_0^2 |y|^2 dxdt + \sum_{i=1}^2 \iint_Q \rho_0^2 |p^i|^2 dxdt \right).
\end{align}
Analogously, multiplying $\rho_2^2 p^i$ to the equation $(\ref{eq18ch5})_2$ and integrating in $Q$, we have 
\begin{align}
\label{estim-ad2ch5}
\underset{[0,T]}{\text{sup}} (\rho_2^2 \|p^i(t)\|^2) &+ \iint_Q \rho_2^2 |p^i_x|^2 dxdt \nonumber\\
&\leq C \left(\iint_Q \rho^2 |G_i|^2 dxdt + \iint_Q \rho_0^2 |y|^2 dxdt + \iint_Q \rho_0^2 |p^i|^2 dxdt \right).
\end{align}
Now, multiplying by $\rho_3^2 y_t$ the equation $(\ref{eq18ch5})_1$ and integrating in $I$, we have
\begin{align*}
\frac{1}{2} & \frac{d}{dt} \left(\int_I a(0,t,x)\rho_3^2 |y_x|^2 dx \right) +  \int_I \rho_3^2 |y_t|^2 dx \leq  \int_I a(0,t,x) \rho_{3,t} \rho_3 |y_x|^2 dx + \int_I D_2 a(0,t,x) \rho_3^2 |y_x|^2 dx \\
&+ \epsilon \int_I \rho_3^2 |y_t|^2 dx + C_\epsilon \left(\int_I \rho_3^2 |G|^2 dx + \int_{\mathcal{O}} \rho_3^2 |f|^2 dx + \int_I \rho_3^2 |y|^2 dx + \int_I \rho_3^2 |y_x|^2 dx + \sum_{i=1}^2 \int_I \rho_3^2 |p^i|^2 dx \right).
\end{align*}
This is
\begin{align*}
\frac{d}{dt} &\left(\int_I a(0,t,x)\rho_3^2 |y_x|^2 dx \right) +  \int_I \rho_3^2 |y_t|^2 dx \\
&\leq C \left(\int_I \rho^2 |G|^2 dx + \int_{\mathcal{O}} \rho_1^2 |f|^2 dx + \int_I \rho_0^2 |y|^2 dx + \sum_{i=1}^2 \int_I \rho_0^2 |p^i|^2 dx + \int_I \rho_2^2 |y_x|^2 dx\right).
\end{align*}
Integrating from $0$ to $t$ and using $(\ref{estim-ad1ch5})$, we have
\begin{align}
\label{estim-ad3ch5}
\underset{[0,T]}{\text{sup}} (\rho_3^2(t) \|y_x(t)\|^2) + \iint_Q \rho_3^2 |y_t|^2 dxdt &\leq C \left(\|y_0\|_{H_0^1(I)}^2 + \iint_Q \rho^2 |G|^2 dxdt + \iint_{\mathcal{O} \times (0,T)} \rho_1^2 |f|^2 dxdt \right. \nonumber\\
&+ \left. \iint_Q \rho_0^2 |y|^2 dxdt + \sum_{i=1}^2 \iint_Q \rho_0^2 |p^i|^2 dxdt \right).
\end{align}
Now, multiplying by $-\rho_3^2 y_{xx}$ the equation $(\ref{eq18ch5})_1$ and integrating in $I$, we have
\begin{align*}
\int_I a(0,t,x) \rho_3^2 |y_{xx}|^2 dx &\leq - \int_I D_3 a(0,t,x) \rho_3^2 y_x y_{xx}dx + \epsilon \int_I \rho_3^2 |y_{xx}|^2 dx \\
&+ C_\epsilon \left(\int_I \rho_3^2 |G|^2 dx + \int_\mathcal{O} \rho_3^2 |f|^2 dx + \int_I \rho_3^2 |y_t|^2dx \right.\\
&+ \left. \int_I \rho_3^2|y|^2 dx + \int_I \rho_3^2|y_x|^2 dx + \sum_{i=1}^2 \int_I \rho_3^2 |p^i|^2 dx \right).
\end{align*}
then
\begin{align*}
\int_I \rho_3^2 |y_{xx}|^2 dx &\leq  C \left(\int_I \rho^2 |G|^2 dx + \int_\mathcal{O} \rho_1^2 |f|^2 dx + \int_I \rho_3^2 |y_t|^2dx \right.\\
&+ \left. \int_I \rho_0^2|y|^2 dx + \int_I \rho_2^2|y_x|^2 dx + \sum_{i=1}^2 \int_I \rho_0^2 |p^i|^2 dx \right).
\end{align*}
Integrating from $0$ to $t$ and using $(\ref{estim-ad1ch5})$ and $(\ref{estim-ad3ch5})$, we get
\begin{align}
\label{estim-ad4ch5}
\iint_Q \rho_3^2 |y_{xx}|^2 dxdt &\leq C \left(\|y_0\|_{H_0^1(I)}^2 + \iint_Q \rho^2 |G|^2 dxdt + \iint_{\mathcal{O} \times (0,T)} \rho_1^2 |f|^2 dxdt \right. \nonumber\\
&+ \left. \iint_Q \rho_0^2|y|^2 dxdt + \sum_{i=1}^2 \iint_Q \rho_0^2 |p^i|^2 dxdt \right)
\end{align}
Analogously, multiplying $\rho_3^2 p^i_t$ to the equation $(\ref{eq18ch5})_2$ and integrating in $Q$, we have 
\begin{align}
\label{estim-ad5ch5}
\underset{[0,T]}{\text{sup}} (\rho_3^2 \|p^i_x(t)\|^2) &+ \iint_Q \rho_3^2 |p^i_t|^2 dxdt \nonumber\\
&\leq C \left(\iint_Q \rho^2 |G_i|^2 dxdt + \iint_Q \rho_0^2 |y|^2 dxdt + \iint_Q \rho_0^2 |p^i|^2 dxdt \right)
\end{align}
and also multiplying $-\rho_3^2 p^i_{xx}$ to the equation
$(\ref{eq18ch5})_2$ and integrating in $Q$, we have
\begin{align}
\label{estim-ad6ch5}
\iint_Q \rho_3^2 |p^i_{xx}|^2 dxdt &\leq C \left(\iint_Q \rho^2 |G_i|^2 dxdt + \iint_Q \rho_0^2 |y|^2 dxdt + \iint_Q \rho_0^2 |p^i|^2 dxdt \right)
\end{align}
From $(\ref{estim-ad1ch5})-(\ref{estim-ad6ch5})$, we have $(\ref{eq37ch5})$.
\end{proof}

\begin{propposition}
\label{prop6ch5}
Let the hypotheses in Proposition $\ref{prop3ch5}$ be satisfied and let $f$ and $(y,p^1,p^2)$ satisfy $(\ref{eq31ch5})$. Then one has
\begin{equation}
\label{eq38ch5}
\begin{aligned}
&\underset{[0,T]}{\text{sup}} ({\rho}_4^2(t) \|y_t(t)\|^2) + \iint_Q \rho_4^2 |y_{xt}|^2 dxdt \\
+&\underset{[0,T]}{\text{sup}} ({\rho}_5^2(t) \|y_{xt}(t)\|^2) + \iint_Q {\rho}_5^2 (|y_{tt}|^2 + |y_{xxt}|^2) dxdt +\underset{[0,T]}{\text{sup}} ({\rho}_5^2(t) \|y_{xx}(t)\|^2) \\
\leq &C \left(\|y_0\|_{H^3_0(I)}^2 + \|G(0)\|_{H_0^1(I)}^2 + \iint_Q \rho^2 |G|^2 dxdt + \iint_Q \rho_3^2 |G_t|^2 dxdt\right.\\
&+ \sum_{i=1}^2 \iint_Q \rho^2 |{G}_i|^2 dxdt  + \iint_{\mathcal{O}\times(0,T)} \rho_1^2 |f|^2 dxdt + \left. \sum_{i=1}^2 \iint_Q \rho_0^2 |p^i|^2 dxdt + \iint_Q \rho_0^2 |y|^2 dxdt \right).
\end{aligned}
\end{equation}
\end{propposition}

\begin{proof}
We know that 
\begin{equation}
\label{derivative_tch5}
y_{tt} - (D_2 a(0,t,x) y_x + a(0,t,x) y_{xt})_x + D_1F(0,0) y_t + D_2F(0,0) y_{xt} = f_t 1_{\mathcal{O}} - \frac{1}{\mu_1} p^1_t 1_{\mathcal{O}_1} - \frac{1}{\mu_2} p^2_t 1_{\mathcal{O}_2} + G_{t}.
\end{equation}
From $(\ref{derivative_tch5})$ multiplying by ${\rho}_4^2 y_t$ and integrating in $I$, we have
\begin{align*}
\frac{1}{2}\frac{d}{dt} &\left({\rho}_4^2(t) \|y_t(t)\|^2 dx\right) + \int_I a(0,t,x){\rho}_4^2 |y_{xt}|^2 dx \leq - \int_I D_2 a(0,t,x) \rho_4^2 y_x y_{xt} dx + \int_I \rho_4 \rho_{4,t} |y_t|^2 dx\\
&+ C \left(\int_I {\rho}_4^2 |G_{t}|^2 dx  + \int_\mathcal{O} {\rho}_4^2 |f_t|^2 dx  + \int_I {\rho}_4^2 |y_t|^2 dx + \sum_{i=1}^2 \int_I {\rho}_4^2 |p^i_t|^2 dx  \right).
\end{align*}
Integrating from $0$ to $t$ and using $(\ref{control-estimatech5})$ and $(\ref{eq37ch5})$, we have
\begin{equation}
\label{ad1ch5}
\begin{aligned}
\underset{[0,T]}{\text{sup}} ({\rho}_4^2(t) \|y_t(t)\|^2  ) &+ \iint_Q {\rho}_4^2 |y_{xt}|^2 dxdt \leq C \left( \|y_t(0)\|^2 + \iint_Q \rho_3^2 |G_t|^2 dxdt  \right.\\
&+ \iint_{\mathcal{O}\times (0,T)} {\rho}_1^2 |f|^2 dxdt  + \iint_Q {\rho}_0^2 |y|^2 dxdt + \left.\sum_{i=1}^2 \iint_Q {\rho}_0^2 |p^i|^2 dxdt \right).
\end{aligned}
\end{equation}
We get easily that
$$\|y_t(0)\| \leq C( \|y(0)\|_{H^2(I)} + \|f(0)\|_{L^2(\mathcal{O})} + \|p^1(0)\| + \|p^2(0)\| + \|G(0)\|).$$
Since $\rho_3 f 1_\mathcal{O}$,$\rho_3 G \in H^1(0,T;L^2(I))$, in $(\ref{ad1ch5})$ one has
\begin{equation}
\label{estim-ad7ch5}
\begin{aligned}
\underset{[0,T]}{\text{sup}} ({\rho}_4^2(t) \|y_t(t)\|^2  ) &+ \iint_Q {\rho}_4^2 |y_{xt}|^2 dxdt \leq C \left( \|y_0\|_{H^2(I)}^2 + \iint_Q \rho^2 |G|^2 dxdt   \right.\\
&+ \iint_Q \rho_3^2 |G_t|^2 dxdt + \sum_{i=1}^2 \iint_Q \rho^2 |G_i|^2 dxdt + \iint_{\mathcal{O}\times (0,T)} {\rho}_1^2 |f|^2 dxdt \\
&+ \iint_Q {\rho}_0^2 |y|^2 dxdt + \left.\sum_{i=1}^2 \iint_Q {\rho}_0^2 |p^i|^2 dxdt \right).
\end{aligned}
\end{equation}
From $(\ref{derivative_tch5})$ multiplying by ${\rho}_5^2 y_{tt}$ and integrating in $I$, we have
\begin{align*}
\int_I {\rho}_5^2 |y_{tt}|^2 dx + \frac{1}{2}\frac{d}{dt} \left(\int_I a(0,t,x){\rho}_5^2 |y_{xt}|^2 dx \right) &\leq \int_I a(0,t,x) {\rho}_5 {\rho}_{5,t} |y_{xt}|^2 dx + \frac{1}{2} \int_I D_2 a(0,t,x) \rho_5^2 |y_{xt}|^2 dx\\
&+ \epsilon \int_I {\rho}_5^2 |y_{tt}|^2 dx + C_\epsilon \left( \int_I {\rho}_5^2 |G_{t}|^2 dx +  \int_\mathcal{O} {\rho}_5^2 |f_t|^2 dx \right.\\
&\left. + \int_I \rho_5^2 |y_t|^2 dx + \int_I \rho_5^2 |y_{xt}|^2 dx + \sum_{i=1}^2 \int_I {\rho}_5^2 |p_t^i|^2 dx  \right)
\end{align*}
Integrating from $0$ to $t$, using $(\ref{control-estimatech5})$, $(\ref{estim-ad3ch5})$, $(\ref{estim-ad5ch5})$ and $(\ref{estim-ad7ch5})$ we deduce
\begin{align*}
\iint_Q {\rho}_5^2 |y_{tt}|^2 dxdt &+ \underset{[0,T]}{\text{sup}}\left({\rho}_5^2(t) \|y_{xt}(t)\|^2 \right) \leq C \left( \|y_{xt}(0)\|^2 + \iint_Q {\rho}_3^2 |G_{t}|^2 dxdt \right.\\
&+ \iint_{\mathcal{O}\times (0,T)} {\rho}_1^2 |f|^2 dxdt + \iint_Q {\rho}_0^2 |y|^2 dxdt   +\left. \sum_{i=1}^2 \iint_Q {\rho}_0^2 |p^i|^2 dxdt  \right).
\end{align*}
We see easily that
$$\|y_{x,t}(0)\| \leq C( \|y_0\|_{H^3(I)} + \|f_x(0)\|_{L^2(\mathcal{O})} + \|p^1_x(0)\| + \|p^2_x(0)\| + \|G_x(0)\|).$$
Then using $(\ref{estim-ad5ch5})$ and $(\ref{control-estimatech5})$, we deduce
\begin{equation}
\label{estim-ad8ch5}
\begin{aligned}
\iint_Q {\rho}_5^2 |y_{tt}|^2 dxdt &+ \underset{[0,T]}{\text{sup}} \left({\rho}_5^2(t) \|y_{xt}(t)\|^2 \right)  \leq C \left(\|y_0\|^2_{H^3(I)} + \|G(0)\|_{H_0^1(I)}^2  \right.\\
&+ \iint_Q {\rho}_3^2 |G_{t}|^2 dxdt + \sum_{i=1}^2 \iint_Q \rho^2 |G_i|^2 dxdt   + \iint_Q {\rho}_1^2 |f|^2 dxdt \\
 & + \iint_Q {\rho}_0^2 |y|^2 dxdt +\left. \sum_{i=1}^2 \iint_Q {\rho}_0^2 |p^i|^2 dxdt  \right).
\end{aligned}
\end{equation}
Analogously from $(\ref{derivative_tch5})$ multiplying by $-{\rho}_5^2 y_{xxt}$ and integrating in $Q$, we have
\begin{equation}
\label{estim-ad9ch5}
\begin{aligned}
\underset{[0,T]}{\text{sup}} \left({\rho}_5^2(t) \|y_{xt}(t)\|^2 \right) &+ \iint_Q {\rho}_5^2 |y_{xxt}|^2 dxdt \leq C \left(\|y_0\|^2_{H^3(I)} + \|G(0)\|_{H_0^1(I)}^2  \right.\\
&+ \iint_Q {\rho}_3^2 |G_{t}|^2 dxdt + \sum_{i=1}^2 \iint_Q \rho^2 |G_i|^2 dxdt   + \iint_Q {\rho}_1^2 |f|^2 dxdt \\
 & + \iint_Q {\rho}_0^2 |y|^2 dxdt +\left. \sum_{i=1}^2 \iint_Q {\rho}_0^2 |p^i|^2 dxdt  \right).
\end{aligned}
\end{equation}
Also, from $(\ref{eq18ch5})_1$ multiplying by $-\rho_5^2 y_{xxt}$ and integrating in $I$, we have
\begin{align*}
\int_I \rho_5^2 |y_{xt}|^2 dx + \frac{1}{2} \frac{d}{dt} \left(\int_I a(0,t,x) \rho_5^2 |y_{xx}|^2 dx\right) &\leq  \int_I a(0,t,x)\rho_5 \rho_{5,t} |y_{xx}|^2 dx + \frac{1}{2} \int_I D_2 a(0,t,x) \rho_5^2 |y_{xx}|^2 dx\\
&+ C \left( \int_{\mathcal{O}} \rho_5^2 |f|^2 dx + \int_I \rho_5^2 |y|^2 dx + \int_I \rho_5^2 |y_x|^2 dx \right.\\
&+ \left.\int_I \rho_5^2 |y_{xxt}|^2 dx  + \sum_{i=1}^2 \int_I \rho_5^2 |p^i|^2 dx + \int_I \rho_5^2 |G|^2 dx \right)
\end{align*}
whence
\begin{equation}
\label{estim-ad10ch5}
\begin{aligned}
\iint_Q {\rho}_5^2 |y_{xt}|^2 dxdt &+ \underset{[0,T]}{\text{sup}} \left({\rho}_5^2(t) \|y_{xx}(t)\|^2 \right) \leq C \left(\|y_0\|^2_{H^3(I)} + \|G(0)\|_{H_0^1(I)}^2  \right.\\
&+ \iint_Q {\rho}_3^2 |G_{t}|^2 dxdt + \sum_{i=1}^2 \iint_Q \rho^2 |G_i|^2 dxdt   + \iint_Q {\rho}_1^2 |f|^2 dxdt \\
 & + \iint_Q {\rho}_0^2 |y|^2 dxdt +\left. \sum_{i=1}^2 \iint_Q {\rho}_0^2 |p^i|^2 dxdt  \right).
\end{aligned}
\end{equation}

Gathering $(\ref{estim-ad7ch5})-(\ref{estim-ad10ch5})$ , we have $(\ref{eq38ch5})$.
\end{proof}

\section{Null Controllability for Nonlinear Optimality System (\ref{EC7ch5})}\label{s5}

To finalize Theorem \ref{teo1ch5}, we will prove the Nash Quasi-Equilibrium for the optimality system (\ref{EC7ch5}), but this result is equivalent to prove the null controllability. So, in this section we use the Right Inverse Function theorem for Banach spaces to conclude the proof.\\

Let us introduce the space
\begin{align*}
Y:=\left\{\right.&(y,p^1,p^2,f); \rho_0 y, \rho_0 p^i \in L^2(Q); \rho_1 f \in L^2(\mathcal{O}\times(0,T));\\
&\rho (y_t - (a(0,t,x)y_{x})_x + D_1F(0,0) y + D_2F(0,0) y_x - f {1}_{\mathcal{O}} + \frac{1}{\mu_1} p^1 1_{\mathcal{O}_1} + \frac{1}{\mu_2} p^2 1_{\mathcal{O}_2}),\\
&\rho_3 (y_{tt} - ((D_2 a(0,t,x) y_x + a(0,t,x))y_{xt})_x + D_1F(0,0) y_t + D_2F(0,0) y_{xt} - f_t {1}_{\mathcal{O}}) \in L^2(Q),\\
&\rho (-p^i_t - (a(0,t,x)p^i_{x})_x + D_1 F(0,0) p^i - D_2 F(0,0) p^i_x - \alpha_i y 1_{\mathcal{O}_{i,d}})\in L^2(Q);\\
&y(0) \in H^3(I) \cap  H_0^1(I)\left. \right\}
\end{align*}
with norm
\begin{align*}
\|(&y,p^1,p^2,f)\|_Y^2 := \|\rho_0 y\|^2_{L^2(Q)}+ \sum_{i=1}^2\|\rho_0 p^i\|^2_{L^2(Q)}+\|\rho_1 f\|^2_{L^2(\mathcal{O}\times(0,T))}\\
&+\|\rho (y_t - (a(0,t,x)y_{x})_x + D_1F(0,0) y + D_2F(0,0) y_x - f {1}_{\mathcal{O}} + \frac{1}{\mu_1} p^1 1_{\mathcal{O}_1} + \frac{1}{\mu_2} p^2 1_{\mathcal{O}_2})\|^2_{L^2(Q)} \\
&+\|\rho_3 (y_{tt} - (D_2 a(0,t,x)y_x + a(0,t,x)y_{xt})_x + D_1F(0,0) y_t + D_2F(0,0) y_{xt} - f_t {1}_{\mathcal{O}})\|^2_{L^2(Q)} +\|y(0)\|^2_{H^3(I)}\\
&+ \sum_{i=1}^2 \|\rho (-p_t^i - (a(0,t,x)p^i_{x})_x + D_1 F(0,0) p^i - D_2 F(0,0) p^i_x + \alpha_i y 1_{\mathcal{O}_{i,d}})\|^2_{L^2(Q)}.
\end{align*}
It is clear that $Y$ is a Hilbert space equipped with the norm $\|\cdot\|_Y$.

Let $L^2(\rho^2;Q)$  be the Hilbert space formed by the measurable functions $w=w(x,t)$ such that $\rho w \in L^2(Q)$, i.e. 
$$\|w\|^2_{L^2(\rho^2;Q)} := \iint_Q \rho^2 |w|^2 dxdt < + \infty$$
and $F:= \left\{ g \in L^2(Q); \rho g, \rho_3 g_t \in L^2(Q), g(0) \in H^1_0(I) \right\}$.

Let us introduce the Hilbert space
\begin{align*}
Z &:= F \times (L^2(\rho^2;Q))^2 \times (H^3(I) \cap H^1_0(I))
\end{align*}
with norm
\begin{align*}
\|(G,G_1,G_2,y_0)\|_Z^2 &:= \|\rho G\|^2_{L^2(Q)} +\|\rho_3 G_{t}\|^2_{L^2(Q)} + \|G(0)\|_{H_0^1(I)}^2 + \|\rho G_1\|^2_{L^2(Q)}  \\
&+ \|\rho G_2\|^2_{L^2(Q)} + \|y_0\|^2_{H^3(I)}.
\end{align*}
\begin{observation}
\label{obs1ch5}
Notice that, if $(y,p^1,p^2,f)\in Y$, in view of Propositions $\ref{prop5ch5}$ and $\ref{prop6ch5}$, one has 
\begin{align*}
&\underset{[0,T]}{\text{sup}} (\rho_3^2(t) \|y_x(t)\|^2) + \underset{[0,T]}{\text{sup}} (\rho_3^2(t) \|p_x(t)\|^2) + \iint_Q \rho_3^2(|y_{t}|^2+|y_{xx}|^2+|p_t|^2+|p_{xx}|^2 )dxdt \\
&+ \underset{[0,T]}{\text{sup}} ({\rho}_4^2(t) \|y_t(t)\|^2) + \iint_Q {\rho}_4^2 |y_{xt}|^2 dxdt +\underset{[0,T]}{\text{sup}} ({\rho}_5^2(t) \|y_{xx}(t)\|^2) + \iint_Q {\rho}_5^2|y_{xxt}|^2 dxdt\\
&\leq C \|(y,p^1,p^2,f)\|^2_Y. 
\end{align*}  
\end{observation}
Let us define the mapping $\mathcal{A}: Y \rightarrow Z$, given by
\begin{align}
\label{eq39ch5}
\mathcal{A}(y,p^1,p^2,f):= (\mathcal{A}_1(y,p^1,p^2,f),\mathcal{A}_2 (y,p^1,p^2,f),\mathcal{A}_3(y,p^1,p^2,f),\mathcal{A}_4(y,p^1,p^2,f))
\end{align}
where
\begin{align*}
\mathcal{A}_1(y,p^1,p^2,f):=&y_t-(a(y_x,t,x)y_x)_x + F(y,y_x)- f {1}_{\mathcal{O}} + \frac{1}{\mu_1} p^1 {1}_{\mathcal{O}_1} + \frac{1}{\mu_2} p^2 {1}_{\mathcal{O}_2},\\
\mathcal{A}_2(y,p^1,p^2,f):=&-p^1_t-((D_1 a(y_x,t,x)y_x + a(y_x,t,x))p_x^1)_x\\
&+ D_1F(y,y_x)p^1 - (D_2F(y,y_x)p^1)_x - \alpha_1 y {1}_{\mathcal{O}_{1,d}},\\
\mathcal{A}_3(y,p^1,p^2,f):=&-p^2_t-((D_1 a(y_x,t,x)y_x + a(y_x,t,x))p_x^2)_x\\
&+ D_1F(y,y_x)p^2 - (D_2F(y,y_x)p^2)_x - \alpha_2 y {1}_{\mathcal{O}_{2,d}},\\
\mathcal{A}_4(y,p^1,p^2,f):=&y(0).
\end{align*}

We will use the following lemmas to conclude the desired result.

\begin{lemma}
\label{lem1ch5}
Let $\mathcal{A}: Y \rightarrow Z$ be the mapping defined by $(\ref{eq39ch5})$. Then, $\mathcal{A}$ is well defined and continuous.
\end{lemma}

\begin{proof}
For every $(y,p^1,p^2,f)\in Y$ one has
\begin{align*}
\|\rho & \mathcal{A}_1(y,p^1,p^2,f)\|_{L^2(Q)}^2 \\
= &\iint_Q \rho^2 |y_t-(a(y_x,t,x)y_x)_x + F(y,y_x) - f {1}_{\mathcal{O}} + \frac{1}{\mu_1} p^1 {1}_{\mathcal{O}_1} + \frac{1}{\mu_2} p^2 {1}_{\mathcal{O}_2}|^2 dxdt \\
\leq &C \left(\iint_Q \rho^2 |y_t-(a(0,t,x)y_{x})_x + D_1F(0,0) y + D_2F(0,0) y_x - f {1}_{\mathcal{O}} + \frac{1}{\mu_1} p^1 {1}_{\mathcal{O}_1} + \frac{1}{\mu_2} p^2 {1}_{\mathcal{O}_2}|^2 dxdt \right.\\
& + \iint_Q \rho^2 |a(y_x,t,x)-a(0,t,x)|^2|y_{xx}|^2 dxdt + \iint_Q \rho^2 |D_3 a(y_x,t,x)- D_3 a(0,t,x)|^2 |y_x|^2 dxdt \\
&+ \left. \iint_Q \rho^2 |D_1 a(y_x,t,x)|^2 |y_x|^2|y_{xx}|^2 dxdt + \iint_Q \rho^2 |F(y,y_x) - D_1F(0,0) y - D_2F(0,0) y_x|^2 dxdt \right)\\
= &C(I_1 + I_2 + I_3 +I_4 + I_5).
\end{align*}
We know that
$$I_1 \leq  \|(y,p^1,p^2,f)\|_Y^2 < + \infty.$$
Also
\begin{align*}
I_2 + I_3 + I_4&\leq C \iint_Q \rho^2 |y_x|^2 (|y_{xx}|^2 + |y_x|^2)dxdt\\
&\leq C \int_0^T {\rho}_5^2 \rho_3^2 \|y_{xx}(t)\|^2 (\|y_{xx}(t)\|^2 + \|y_x{t}\|)\ dt\\
&\leq C \left( \underset{[0,T]}{\text{sup}}\ \ {\rho}_5^2(t) \|y_{xx}(t)\|^2 \right) \iint_Q \rho_3^2 (|y_{xx}|^2 + |y_x|^2)dxdt\\
&\leq C  \|(y,p^1,p^2,f)\|_Y^4 < + \infty,
\end{align*}
and
\begin{align*}
I_5 \leq &C \iint_Q \rho^2 (|\nabla F(\tilde{\theta}y,\tilde{\theta}y_x)- \nabla F(0,0)|^2 (|y|^2 + |y_x|^2) dxdt\\
\leq\ &C \iint_Q \rho^2 \tilde{\theta}^2 (|y|^2+|y_x|^2) (|y|^2 + |y_x|^2) dxdt\\
\leq\ &C \int_0^T \rho_5^2 \rho_3^2 (\|y_x(t)\|^2 + \|y_{xx}(t)\|^2) (\|y(t)\|^2 + \|y_x(t)\|^2)\\
\leq\ &C \left\{ \left( \underset{[0,T]}{\text{sup}}\ \ {\rho}_5^2(t) \|y_{xx}(t)\|^2 \right) + \left( \underset{[0,T]}{\text{sup}}\ \ {\rho}_3^2(t) \|y_{x}(t)\|^2 \right) \right\}\cdot \\
&\left( \iint_Q \rho_0^2 |y|^2 dxdt + \iint_Q \rho_2^2 |y_x|^2 dxdt \right)\\
\leq\ &C  \|(y,p^1,p^2,f)\|_Y^4 < + \infty,
\end{align*}
where $\tilde{\theta}:=\tilde{\theta}(x,t) \in (0,1)$.\\ \\
Now
\begin{align*}
\|\rho_3 &\mathcal{A}_{1,t}(y,p^1,p^2,f)\|_{L^2(Q)}\\
=&\iint_Q \rho_3^2 |(y_t - (a(y_x,t,x)y_x)_x + F(y,y_x) - f {1}_{\mathcal{O}} + \frac{1}{\mu_1} p^1 {1}_{\mathcal{O}_1} + \frac{1}{\mu_2} p^2 {1}_{\mathcal{O}_2})_t|^2 dxdt \\
= &\iint_Q \rho_3^2 |y_{tt}- ((D_2 a(y_x,t,x) y_x + a(y_x,t,x))y_{xt})_x + \nabla F(y,y_x)(y,y_x)_t -f_t {1}_{\mathcal{O}} + \frac{1}{\mu_1} p_t^1 {1}_{\mathcal{O}_1} + \frac{1}{\mu_2} p_t^2 {1}_{\mathcal{O}_2}|^2 dxdt \\
\leq &C \left( \iint_Q \rho_3^2 |y_{tt}- ((D_2 a(0,t,x)y_x + a(0,t,x))y_{xt})_x + \nabla F(0,0) (y,y_x)_t -f_t 1_\mathcal{O}|^2 dxdt + \sum_{i=1}^2 \iint_Q \rho_3^2 |p_t^i|^2 dxdt \right.\\
&\left.+ \iint_Q \rho_3^2 |((D_2 a(y_x,t,x)-D_2 a(0,t,x))y_x)_x|^2 dxdt + \iint_Q \rho_3^2 |((a(y_x,t,x)-a(0,t,x))y_{xt})_x|^2 dxdt \right)\\
\leq &C \left(\|(y,p^1,p^2,f)\|_Y^2 + \|(y,p^1,p^2,f)\|_Y^4 + \|(y,p^1,p^2,f)\|_Y^6 \right) < + \infty.
\end{align*}
Analogously, we have
\begin{align*}
\|\mathcal{A}_{1}(y,p^1,p^2,f)(0)\|_{H_0^1(I)}&=\int_I |(y_{xt}(0)-(a(y_x(0,t,0))y_x(0))_{xx} -f_x(0) {1}_{\mathcal{O}} + \frac{1}{\mu} p_x(0) {1}_{\omega}|^2 dx \\
&\leq C (\|(y,p^1,p^2,f)\|_Y^2 + \|(y,p^1,p^2,f)\|_Y^4) < +\infty.
\end{align*}
And finally
\begin{align*}
\|\rho &\mathcal{A}_{i+1}(y,p^1,p^2,f)\|_{L^2(Q)}^2\\
=\ & \iint_Q \rho^2 |-p^i_t-((D_1 a(y_x,t,x)y_x + a(y_x,t,x))p^i_x)_x + D_1F(y,y_x)p^i - (D_2 F(y,y_x) p^i)_x - \alpha_i y {1}_{\mathcal{O}_{i,d}}|^2 dxdt \\
\leq\ &C \left(\iint_Q \rho^2 |-p_t^i - (a(0,t,x)p_{x}^i)_x + D_1F(0,0) p^i - D_2F(0,0) p^i_x - \alpha_i y {1}_{\mathcal{O}_{i,d}}|^2 dxdt \right.\\
&+ \iint_Q \rho^2 |a(y_x,t,x)-a(0)|^2|p_{xx}^i|^2 dxdt +  \iint_Q \rho^2 |D_1 a(y_x,t,x)|^2 |p^i_x|^2|y_{xx}|^2 dxdt \\
&  + \iint_Q \rho^2 |D_{11}^2 a(y_x,t,x)|^2 |y_{xx}|^2 |y_x|^2 |p_x|^2 dxdt + \iint_Q \rho_3^2 |D_1 (F(y,y_x)-F(0,0))|^2 |p^i|^2 dxdt\\
& + \iint_Q \rho_3^2 |D_2 (F(y,y_x)-F(0,0))|^2 |p^i_x|^2 dxdt + \iint_Q \rho_3^2 |D_{12}^2 (F(y,y_x)|^2 |y_x|^2 |p^i|^2 dxdt\\
&\left. + \iint_Q \rho_3^2 |D_2^2 F(y,y_x)|^2 |y_{xx}|^2 |p^i|^2 dxdt\right)\\
\leq\ &C \left( \|(y,p^1,p^2,f)\|_Y^2 + \|(y,p^1,p^2,f)\|_Y^4 + \|(y,p^1,p^2,f)\|_Y^6 \right) < + \infty.
\end{align*}
Consequently, $\mathcal{A}$ takes values em $Z$.\\
That the mapping $\mathcal{A}$ is continuous is easy to prove using similar arguments.
\end{proof}

\begin{lemma}
\label{lem2ch5}
The mapping $\mathcal{A}: Y \rightarrow Z$ is continuously differentiable.
\end{lemma}

\begin{proof}
Let us first prove that $\mathcal{A}$ is $G-$differentiable at any $(y,p^1,p^2,f) \in Y$ and let us compute the $G-$derivative $\mathcal{A}'(y,p^1,p^2,f)$.\\
Thus, let us fix $(y,p^1,p^2,f)$ in $Y$ and let us take $(y',{p^1}',{p^2}',f') \in Y$ and $\lambda>0$.\\
Let us introduce the linear mapping $D\mathcal{A}: Y \rightarrow Z$, with
$$D\mathcal{A}(y,p^1,p^2,f)=D\mathcal{A}=(D\mathcal{A}_1,D\mathcal{A}_2,D\mathcal{A}_3).$$
\begin{align*}
D\mathcal{A}_1(y',{p^1}',{p^2}',f'):=\ &y'_t - ((D_1  a(y_x,t,x) y_x + a(y_x,t,x))y'_x)_x\\
&+ D_1F(y,y_x)y' + D_2F(y,y_x)y_{x}'\\
&- f' {1}_{\mathcal{O}} + \frac{1}{\mu_1} {p^1}' {1}_{\mathcal{O}_1} + \frac{1}{\mu_2} {p^2}' {1}_{\mathcal{O}_2},\\
D\mathcal{A}_{1,t}(y',{p^1}',{p^2}',f'):=\ &y'_{tt} - ((D_1 a(y_x,t,x) y_x + a(y_x,t,x))y'_{x})_{xt} \\
&+ D_{11}^2 F(y,y_x)y_t y' + D_{12}^2 F(y,y_x)y_{xt} y'\\
&+ D_{21}^2 F(y,y_x) y_t y'_x + D^2_{22} F(y,y_x) y_{xt} y_x'\\
&+ D_1 F(y,y_x) y_t' + D_2 F(y,y_x) y_{xt}'\\
&- f'_t {1}_{\mathcal{O}} + \frac{1}{\mu_1} {p^1}'_t {1}_{\mathcal{O}_1} + \frac{1}{\mu_2} {p^2}'_t {1}_{\mathcal{O}_2},\\
D\mathcal{A}_2(y',{p^1}',{p^2}',f'):=\ &-{p^1}'_t - ((D_1 a(y_x,t,x) y_x + a(y_x,t,x)) {p^1}'_x)_x - (D_{11}^2 a(y_x,t,x)y'_x y_x p^1_x)_x\\
&- 2(D_1 a(y_x,t,x)y'_x p^1_x)_x + D_{11}^2 F(y,y_x) y' {p^1} + D_{12}^2 F(y,y_x) y'_x p^1\\
& - (D_{21}^2 F(y,y_x) y' p^1)_x - (D_{22}^2 F(y,y_x) y'_x p^1)_x\\
&+ D_1 F(y,y_x) {p^1}' - (D_2 F(y,y_x) {p^1}')_x - \alpha_1 y' {1}_{\mathcal{O}_1},\\
D\mathcal{A}_3(y',{p^1}',{p^2}',f'):=\ &-{p^2}'_t - ((D_1 a(y_x,t,x) y_x + a(y_x,t,x)) {p^2}'_x)_x - (D_{11}^2 a(y_x,t,x)y'_x y_x p^2_x)_x\\
&- 2(D_1 a(y_x,t,x)y'_x p^2_x)_x + D_{11}^2 F(y,y_x) y' {p^2} + D_{12}^2 F(y,y_x) y'_x p^2\\
& - (D_{21}^2 F(y,y_x) y' p^2)_x - (D_{22}^2 F(y,y_x) y'_x p^2)_x\\
&+ D_1 F(y,y_x) {p^2}' - (D_2 F(y,y_x) {p^2}')_x - \alpha_2 y' {1}_{\mathcal{O}_2},\\
D\mathcal{A}_4(y',{p^1}',{p^2}',f'):=\ &y'(0),
\end{align*}
for all $(y',{p^1}',{p^2}',f') \in Y$.\\
From the definition of the spaces $Y$ and $Z$, it becomes that $D\mathcal{A} \in \mathcal{L}(Y;Z)$.\\
Furthermore, we have
$$\frac{1}{\lambda}[\mathcal{A}_i((y,p^1,p^2,f)+\lambda(y',{p^1}',{p^2}',f'))-\mathcal{A}_i(y,p^1,p^2,f)] \rightarrow D\mathcal{A}_i(y',{p^1}',{p^2}',f')$$
strongly in $L^2(\rho^2;Q)$ for $i=1,2,3,4$ as $\lambda \rightarrow 0$ and
$$\frac{1}{\lambda}[\mathcal{A}_{1,t}((y,p^1,p^2,f)+\lambda(y',{p^1}',{p^2}',f'))-\mathcal{A}_{1,t}(y,p^1,p^2,f)] \rightarrow D\mathcal{A}_{1,t}(y',{p^1}',{p^2}',f')$$
strongly in $L^2(\rho_3^2;Q)$  as $\lambda \rightarrow 0$.\\
Indeed, we denote 
$$\overline{a} := a(y_x,t,x),\ \ \overline{a}_n := a(y_x^n,t,x),\ \ a_\lambda := a(y_x + \lambda y'_x,t,x),$$
$$\overline{a}' := D_1 a(y_x,t,x),\ \ \overline{a}'_n := D_1 a(y_x^n,t,x),\ \ a_\lambda':= D_1 a(y_x + \lambda y'_x,t,x),$$ 
$$\overline{F} := F(y,y_x),\ \ \overline{F}_n:= F(y^n,y^n_x),\ \ F_\lambda := F( y + \lambda y',y_x + \lambda y'_x),$$
$$\overline{F}'_i:= D_i F(y,y_x),\ \ F'_{n,i} := D_i F(y^n,y^n_x),\ \ F_{\lambda,i}' := D_i F(y + \lambda y',y_x + \lambda y'_x),$$ 
and we have
\begin{align*}
&\left\|\frac{1}{\lambda}[\mathcal{A}_1((y,p^1,p^2,f)+\lambda(y',{p^1}',{p^2}',f')) - \mathcal{A}_1(y,p^1,p^2,f)] - D\mathcal{A}_1(y',{p^1}',{p^2}',f')\right\|_{L^2(\rho^2;Q)}^2 \\
\leq &C \left( \iint_Q \rho^2 \left|\left(\left[\frac{a_\lambda - \overline{a}}{\lambda} - \overline{a}' y'_x \right] y_x\right)_x \right|^2 dxdt + \iint_Q \rho^2 |(a_\lambda - \overline{a})y'_x)_x|^2 dxdt \right.\\
&\left. + \iint_Q \rho^2 \left|\frac{F_\lambda - \overline{F}}{\lambda} - \nabla F(y,y_x) (y',y'_x)\right|^2 dxdt \right)\\
\leq &C \left( \iint_Q \rho^2 \left| \frac{a_\lambda - \overline{a}}{\lambda} - \overline{a}' y'_x \right|^2 |y_{xx}|^2 dxdt \right.\\
&+ \iint_Q \rho^2 \left|\frac{a'_\lambda (y_{xx} + \lambda y'_{xx}) - \overline{a}' y_{xx}}{\lambda} - D_{11}^2 a(y_x,t,x) y_{xx} y'_{x} - \overline{a}' y'_{xx} \right|^2 |y_x|^2 dxdt   \\
& +\iint_Q \rho^2 | a_{\lambda} - \overline{a}|^2 |y'_{xx}|^2 dxdt + \iint_Q \rho^2 |a'_{\lambda} (y_{xx} + \lambda y'_{xx}) - \overline{a}' y_{xx}|^2 |y'_x|^2 dxdt\\
& \left. + \iint_Q \rho^2 |\nabla F(y + \tilde{\lambda} y', y_x + \tilde{\lambda} y'_x) - \nabla F(y,y_x)|^2 (|y'|^2 + |y'_x|^2) dxdt \right)\\
\leq &C  \left(\lambda^2 \iint_Q \rho^2 |y'_x|^4 |y_{xx}|^2 dxdt +  \iint_Q \rho^2 \left|\frac{a'_\lambda - \overline{a}'}{\lambda} - D_{11}^2 a(y_x,t,x) y'_{x} \right|^2 |y_{xx}|^2 |y_x|^2 dxdt \right.\\
&+ \lambda^2 \iint_Q \rho^2 \left|y'_x \right|^2 |y'_{xx}|^2 |y_x|^2 dxdt + \lambda^2 \iint_Q \rho^2 |y'_x|^2 |y'_{xx}|^2 dxdt \\
& \left. + \lambda^2 \iint_Q \rho^2 |y_{xx}|^2 |y'_x|^4 dxdt + \lambda^2 \iint_Q \rho^2 |y'_{xx}|^2 |y'_x|^2 dxdt+ \lambda^2 \iint_Q \rho^2 (|y'|^4 + |y'_x|^4) \right),
\end{align*}
where $\tilde{\lambda}:= \tilde{\lambda}(x,t) \in (0,\lambda)$.\\
Using Observation $\ref{obs1ch5}$ and Lebesgue's Theorem, we find that
$$\frac{1}{\lambda}[\mathcal{A}_1((y,p^1,p^2,f)+\lambda(y',{p^1}',{p^2}1,f'))-\mathcal{A}_1(y,p^1,p^2,f)] \rightarrow D\mathcal{A}_1(y',{p^1}',{p^2}',f').$$
Similarly
\begin{align*}
&\left\|\frac{1}{\lambda} \Bigl[\mathcal{A}_{1,t}((y,p^1,p^2,f)+\lambda(y',{p^1}',{p^2}',f'))-\mathcal{A}_{1,t}(y,p^1,p^2,f)\Bigr] - D\mathcal{A}_{1,t}(y',{p^1}',{p^2}',f') \right\|_{L^2(\rho_3^2;Q)}^2 \\
&\leq C \left( \iint_Q \rho_3^2 \left|\left(\left[\frac{a_\lambda - \overline{a}}{\lambda} - \overline{a}' y'_x \right] y_x\right)_{xt} \right|^2 dxdt + \iint_Q \rho_3^2 |(a_\lambda - \overline{a})y'_x)_{xt}|^2 dxdt \right. \\
& + \iint_Q \rho_3^2 \left|\frac{F'_{\lambda,1}-\overline{F}'_1}{\lambda} - D_{11}^2 F(y,y_x) y' - D_{12}^2 F(y,y_x) y'_x \right|^2 |y_t|^2 dxdt \\
& + \iint_Q \rho_3^2 \left|\frac{F'_{\lambda,2}-\overline{F}'_2}{\lambda} - D_{21}^2 F(y,y_x) y' - D_{22}^2 F(y,y_x) y'_x \right|^2 |y_{xt}|^2 dxdt \\
&\left. + \iint_Q \rho_3^2 |F'_{\lambda,1} - \overline{F}'_1|^2 |y'_t|^2 dxdt + \iint_Q \rho_3^2 |F'_{\lambda,2} - \overline{F}'_2|^2 |y'_{xt}|^2 dxdt \right).
\end{align*}
Using Observation $\ref{obs1ch5}$ and Lebesgue's Theorem, we find that
$$\frac{1}{\lambda}[\mathcal{A}_{1,t}((y,p^1,p^2,f)+\lambda(y',{p^1}',{p^2}',f'))-\mathcal{A}_{1,t}(y,p^1,p^2,f)] \rightarrow D\mathcal{A}_{1,t}(y',{p^1}',{p^2}',f').$$
Similarly
\begin{align*}
&\left\|\frac{1}{\lambda} \Bigl[\mathcal{A}_1((y,p^1,p^2,f)+\lambda(y',{p^1}',{p^2}',f'))(0)-\mathcal{A}_1(y,p^1,p^2,f)(0)\Bigr] - D\mathcal{A}_1(y',{p^1}',{p^2}',f')(0)\right\|_{H^1_0(I)}^2 \\
&\leq C \left( \int_I \Bigl|\Bigl[ \left(\frac{1}{\lambda}(a(y_x(0)+\lambda y'_x(0),t,0)-a(y_x(0),t,0))- D_1 a(y_x(0),t,0) y'_x(0) \right)y_x(0) \Bigr]_{xx} \Bigr|^2 dx \right. \\
&+  \int_I \Bigl|\Bigl[(a(y_x(0)+\lambda y'_x(0),t,0)-a(y_x(0),t,0))y_x(0)\Bigr]_{xx} \Bigr|^2 dx \\
&\left. + \iint_Q \rho^2 \left|\left[\frac{F(y+\lambda y',y_x + \lambda y'_x)(0) - F(y,y_x)(0)}{\lambda} - \nabla F(y,y_x)(0) (y',y'_x)(0)\right]_x \right|^2 dxdt \right).
\end{align*}
Using Observation $\ref{obs1ch5}$  and Lebesgue's Theorem, we find that
$$\frac{1}{\lambda}[\mathcal{A}_1((y,p^1,p^2,f)+\lambda(y',{p^1}',{p^2}',f'))(0)-\mathcal{A}_1(y,p^1,p^2,f)(0)] \rightarrow D\mathcal{A}_1(y',{p^1}',{p^2}',f')(0).$$
and finally
\begin{align*}
\Bigl\|\frac{1}{\lambda}& \Bigl[\mathcal{A}_i((y,p^1,p^2,f) +\lambda(y',{p^1}',{p^2}',f'))-\mathcal{A}_2(y,p^1,p^2,f) \Bigr] - D\mathcal{A}_2(y',{p^1}',{p^2}',f')\Bigr\|_{L^2(\rho^2;Q)}^2 \\
&\leq C \left( \iint_Q \rho^2 \left|\left(\left(\frac{a'_\lambda - \overline{a}'}{\lambda}- D_{11}^2 a(y_x,t,x) y'_x \right) y_x p^i_x \right)_x \right|^2 dxdt+ \iint_Q \rho^2 \left|\left(\left(\frac{a_\lambda - \overline{a}}{\lambda}-\overline{a}' y'_x \right)p^i_x \right)_x \right|^2 dxdt \right.\\
&+ \iint_Q \rho^2 |((a'_\lambda-\overline{a}')y_x {p^i}'_x)_x|^2 dxdt + \iint_Q \rho^2 |((a'_\lambda-\overline{a}')y'_x p^i_x)_x|^2 |{p^i}'|^2 dxdt +  \iint_Q \rho^2 |((a'_\lambda y_x {p^i}'_x)_x|^2 dxdt\\
&+ \iint_Q \rho^2 \left|\frac{F'_{\lambda,1}-\overline{F}'_1}{\lambda} - D_{11}^2 F(y,y_x) y' p - D_{12}^2 F(y,y_x) y'_x \right|^2 |p^i|^2 dxdt \\
& + \iint_Q \rho^2 \left|\frac{(F'_{\lambda,2}-\overline{F}'_2)_x}{\lambda} - (D_{21}^2 F(y,y_x) y' p^i)_x - (D_{22}^2 F(y,y_x) y'_x p^i)_x \right|^2 dxdt \\
&\left. + \iint_Q \rho^2 |F'_{\lambda,1} - \overline{F}'_1|^2 |{p^i}'|^2 dxdt + \iint_Q \rho_3^2 |((F'_{\lambda,2} - \overline{F}'_2) {p^i}')_x|^2 dxdt \right)\\
\end{align*}
Using Observation $\ref{obs1ch5}$  and Lebesgue's Theorem, we find that
$$\frac{1}{\lambda}[\mathcal{A}_i((y,p^1,p^2,f)+\lambda(y',{p^1}',{p^2}',f'))-\mathcal{A}_i(y,p^1,p^2,f)] \rightarrow D\mathcal{A}_i(y',{p^1}',{p^2}',f').$$

Then $\mathcal{A}$ is $G-$differentiable at any $(y,p^1,p^2,f) \in Y$, with a $G-$derivative 
$$\mathcal{A}'(y,p^1,p^2,f)=D\mathcal{A}$$
Now, we shall prove that the mapping $(y,p^1,p^2,f) \mapsto \mathcal{A}'(y,p^1,p^2,f)$ is continuous from $Y$ into $\mathcal{L}(Y;Z)$. As a consequence, in view of classical results, we will have that $\mathcal{A}$ is not only $G-$differentiable but also $F-$differentiable and $C^1$.\\
Thus, let us assume that $(y^n,p^{1,n},p^{2,n},f^n) \rightarrow (y,p^1,p^2,f)$ in $Y$ and let us check that
\begin{align}
\label{eq40ch5}
\|(D\mathcal{A}(y^n,p^{1,n},p^{2,n},f^n) - D\mathcal{A}(y,p^1,p^2,f))(y',{p^1}',{p^2}',f')\|_Y^2 \leq \epsilon_n \|(y',{p^1}',{p^2}',f')\|_Y^2
\end{align}
for all $(y',{p^1}',{p^2}',f') \in Y$, for some $\epsilon_n \rightarrow 0$.\\
The following holds, using the Observation $\ref{obs1ch5}$  and Lebesgue's Theorem
\begin{align*}
\|(D&\mathcal{A}_1 (y^n,p^{1,n},p^{2,n},f^n) - D\mathcal{A}_1 (y,p^1,p^2,f))(y',{p^1}',{p^2}',f')\|_{L^2(\rho^2;Q)}^2\\
&\leq C \left( \iint_Q \rho^2 |(\overline{a}'_n y'_x y_{n,x})_x - (\overline{a}' y'_x y_x)_x + (\overline{a}'_n y'_x)_x - (\overline{a}' y'_x)_x|^2 dxdt \right.\\
&+ \iint_Q \rho^2 |D_1 F(y^n,y^n_x) - D_1 F(y,y_x)|^2 |y'|^2 dxdt\\
&\left.  + \iint_Q \rho^2 |D_2 F(y^n,y^n_x) - D_2 F(y,y_x)|^2 |y'_x|^2 dxdt \right)\\
&\leq C \left( \iint_Q \rho^2 |((\overline{a}'_n - \overline{a}') y'_x y_x)_x|^2 dxdt + \iint_Q \rho^2 |(\overline{a}'_n y'_x (y_{n,x} - y_x))_x|^2 dxdt \right.\\
&+ \left. \iint_Q \rho^2 |((\overline{a}'_n - \overline{a}') y'_x)_x|^2 dxdt + \iint_Q \rho^2 (|y^n - y|^2 + |y^n_x - y_x|^2) (|y'|^2 + |y'_x|^2) dxdt \right)\\
& \leq \epsilon_{1,n} \|(y',{p^1}',{p^2}',f')\|_Y^2
\end{align*}
where
\begin{align*}
\epsilon_{1,n} &:= C (1+ \|(y^n,p^{1,n},p^{2,n},f^n)\|_Y^2 + \|(y,p^1,p^2,f)\|_Y^2) \|(y^n,p^{1,n},p^{2,n},f^n)-(y,p^1,p^2,f)\|_Y^2  
\end{align*}
For the other component, similar arguments lead to the same conclusion.
\begin{align*}
\|(D &\mathcal{A}_{1,t}(y^n,p^{1,n},p^{2,n},f^n) - D\mathcal{A}_{1,t} (y,p^1,p^2,f))(y',{p^1}',{p^2}',f')\|_{L^2(\rho_3^2;Q)}^2\\
&\leq C \left( \iint_Q \rho_3^2 |((\overline{a}'_n - \overline{a}') y'_x y_x)_{xt}|^2 dxdt + \iint_Q \rho_3^2 |(\overline{a}'_n y'_x (y_{n,x} -  y_x))_{xt}|^2 dxdt \right.\\
&+ \iint_Q \rho_3^2 |((\overline{a}'_n - \overline{a}') y'_x)_{xt}|^2 dxdt \\
&+ \iint_Q \rho_3^2 |D_{11}^2 F(y^n,y^n_x) y^n_t - D_{11}^2 F(y,y_x) y_t|^2 |y'|^2 dxdt\\
&+ \iint_Q \rho_3^2 |D_{12}^2 F(y^n,y^n_x) y^n_{xt}-D_{12}^2 F(y,y_x) y_{xt}|^2 |y'|^2 dxdt\\
&+ \iint_Q \rho_3^2 |D_{21}^2 F(y^n,y^n_x) y^n_{t} - D_{21}^2 F(y,y_x)) y_{t}|^2 |y'_x|^2 dxdt\\
&+ \iint_Q \rho_3^2 |D_{22}^2 F(y^n,y^n_x) y^n_{xt} -D_{22}^2 F(y,y_x) y_{xt}|^2 |y'_x|^2 dxdt\\
& \left. + \iint_Q \rho_3^2 |F'_{n,1} -\overline{F}'_1|^2 |y'_t|^2 dxdt + \iint_Q \rho_3^2 |F'_{n,2} -\overline{F}'_2|^2 |y'_{xt}|^2 dxdt \right)\\
& \leq \epsilon_{2,n} \|(y',{p^1}',{p^2}',f')\|_Y^2
\end{align*}
where\ $\epsilon_{2,n} \to 0$ as $n \to +\infty$. \\
Similarly
\begin{align*}
\|(D &\mathcal{A}_1(y^n,p^{1,n},p^{2,n},f^n) - D\mathcal{A}_1 (y,p^1,p^2,f))(y',{p^1}',{p^2}',f')(0)\|_{H^1_0(I)}^2\\
&\leq C \int_I \Bigl| \Bigl[ \Bigl(D_1 a(y_x^n(0),t,0)y_x^n(0)-D_1 a(y_x(0),t,0)y_x(0) \Bigr) y'_x(0) \Bigr]_{xx}\Bigr|^2 dx \\
&+ \int_I \Bigl| \Bigl[ \Bigl(a(y_x^n(0),t,0)-a(y_x(0),t,0) \Bigr)y'_x(0) \Bigr]_{xx} \Bigr|^2 dx\\
&+ \int_I \Bigl|[F'_{n,1}(0) {y^n}'(0) - \overline{F}'_1(0) y'(0)]_x \Bigr|^2 dx + \int_I \Bigl|[F'_{n,2}(0) {y^n}'_x(0) - \overline{F}'_2(0) y'_x(0)]_x \Bigr|^2 dx\\
& \leq \epsilon_{3,n} \|(y',{p^1}',{p^2}',f')\|_Y^2
\end{align*}
where\ $\epsilon_{3,n} \to 0$ as $n \to +\infty$. \\
And 
\begin{align*}
\|(D\mathcal{A}_i(y^n,p^{1,n},p^{2,n},f^n) &- D\mathcal{A}_i (y,p^1,p^2,f))(y',{p^1}',{p^2}',f')\|_{L^2(\rho^2;Q)}^2 \\
&\leq C \epsilon_{i+2,n} \|(y',{p^1}',{p^2}',f')\|_Y^2
\end{align*}
where\ $\epsilon_{i+2,n} \to 0$ as $n \to +\infty$.\\
This show that $(\ref{eq40ch5})$ is satisfied.
\end{proof}

\begin{lemma}
\label{lem3ch5}
Let $\mathcal{A}$ be the mapping defined by $(\ref{eq39ch5})$. Then $\mathcal{A}'(0,0,0,0)$ is onto. 
\end{lemma}

\begin{proof}
Let us fix $(G,{G}_1,{G}_2,y_0) \in Z$. From Proposition $\ref{prop3ch5}$ we know that there exists $(y,p^1,p^2,f)$ satisfying $(\ref{eq31'ch5})$, $(\ref{eq31ch5})$ and $(\ref{eq4ch5})$. Consequently, $(y,p^1,p^2,f) \in Y$ and
\begin{align*}
\mathcal{A}'(0,0,0,0)(y,p^1,p^2,f)=\ &(y_t- (a(0,t,x)y_{x})_x + D_1F(0,0)y + D_1F(0,0)y_x - f{1}_{\mathcal{O}} + \frac{1}{\mu_1} p^1 {1}_{\mathcal{O}_1} + \frac{1}{\mu_2} p^2 {1}_{\mathcal{O}_2},\\
&-p^1_t - (a(0,t,x)p^1_{x})_x + D_1 F(0,0) p^1 - D_2 F(0,0) p^1_x -\alpha_1 y {1}_{\mathcal{O}_{1,d}},\\
&-p^2_t - (a(0,t,x)p^2_{x})_x + D_1 F(0,0) p^2 - D_2 F(0,0) p^2_x -\alpha_2 y {1}_{\mathcal{O}_{2,d}},\\
&y(0)) =(G,{G}_1,{G}_2,y_0)
\end{align*}
This end the proof.
\end{proof}

To conclude the proof, we will use the important result
\begin{theorem}[Inverse Function Theorem]
\label{t2prelim}
Let $Y$ and $Z$ be Banach spaces and let $\mathcal{A}: B_r(0) \subset Y \to Z$ be a $C^1$ mapping. Let us assume that the derivative $\mathcal{A}'(0): Y \to Z$ is onto and let us set $\xi_0=\mathcal{A}(0)$. Then there exist $\epsilon>0$, a mapping $W: B_\epsilon(\xi_0)\subset Z \to Y$ and a constant $K>0$ satisfying
$$W(z) \in B_r(0) \ \text{and} \ \ \mathcal{A}(W(z))=z,\ \ \ \forall z \in B_\epsilon(\xi_0),$$
$$\|W(z)\|_Y\leq K\|z-\xi_0\|_Z,\ \ \ \forall z \in B_\epsilon(\xi_0).$$  
\end{theorem}
\begin{proof}
See \cite{ALEKSEEV}.
\end{proof} 

In accordance with Lemmas $\ref{lem1ch5}$, $\ref{lem2ch5}$ and $\ref{lem3ch5}$, we can apply Theorem $\ref{t2prelim}$ and deduce that, there exists $\epsilon>0$, a mapping $W: B_\epsilon(0)\subset Z \to Y$ such that 
$$W(w) \in B_r(0) \ \ \text{and} \ \ \mathcal{A}(W(w))=w\ ,\ \ \ \forall w \in B_\epsilon(0)$$
Taking $(0,-\alpha_1 y_{1,d} {1}_{\mathcal{O}_{1,d}} ,-\alpha_2 y_{2,d} {1}_{\mathcal{O}_{2,d}} ,y_0) \in B_\epsilon(0)$ and 
$$(y,p^1,p^2,f)=W(0,-\alpha_1 y_{1,d} {1}_{\mathcal{O}_{1,d}} ,-\alpha_2 y_{2,d} {1}_{\mathcal{O}_{2,d}} ,y_0) \in Y,$$
we have 
$$\mathcal{A}((y,p^1,p^2,f))=(0,-\alpha_1 y_{1,d} {1}_{\mathcal{O}_{1,d}} ,-\alpha_2 y_{2,d} {1}_{\mathcal{O}_{2,d}} ,y_0)$$
thus, we prove that $(\ref{EC7ch5})$ is null locally controllable at time $T>0$.

\subsection{Nash equilibrium for $(\ref{EC1ch5})$}
Finally, we will prove the Theorem \ref{teo2ch5}, the technique of the proof is based in \cite{Cara1}.\\

Let $f \in L^2(\mathcal{O}\times (0,T))$ be given and let $(v^1,v^2)$ be the associated Nash quasi-equilibrium. For any $s \in \mathbb{R}$ and $w^1,w^2 \in L^2(\mathcal{O}_1 \times (0,T)) $, we have
\begin{align}
\label{eq41ch5}
\langle D_1 J_1(f;v^1+sw^1,v^2),w^2  \rangle &= \alpha_1 \iint_{\mathcal{O}_{1,d} \times (0,T)} (y^s-y_{1,d}) z^s  dxdt \nonumber\\
&+ \mu_1 \iint_{\mathcal{O}_1 \times (0,T)} (v^1 + s w^1)w^2 dxdt
\end{align}  
where
\begin{equation}
\label{EC9ch5}
\left\{
\begin{array}{rl}
&y^s_t - (a(y_x^s,t,x) y^s_x)_x + F(y^s,y_x^s) = f {1}_\mathcal{O} + (v^1+sw^1) {1}_{\mathcal{O}_1} + v^2 1_{\mathcal{O}_1} \ \text{in} \ Q,\\
&y^s(0,t) = y^s(L,t) = 0\ \ \text{ in } \ (0,T),\\
&y^s(0) = y^{0}\ \ \text{in }\ \ I,
\end{array}
\right.
\end{equation}
$z^s$ the derivative of the state $y^s$ with respect to $v^1$ in the direction $w^2$, i. e. the solution to
\begin{equation}
\label{EC10ch5}
\left\{
\begin{array}{rl}
&z^s_t - ((D_1 a(y_x^s,t,x)y_x^s + a(y_x^s,t,x)) z^s_x)_{x} + D_1F(y^s,y^s_x) z^s + D_2 F(y^s,y^s_x) z_x^s = w^2 {1}_{\mathcal{O}_1} \ \text{in} \ Q,\\
&z^s(0,t)=z^s(L,t) = 0\ \ \text{ in } \ (0,T),\\
&z^s(0) = 0\ \ \text{ in }\ \ I.
\end{array}
\right.
\end{equation}
with $y=y^s|_{s=0}$ and $z=z^s|_{s=0}$, then
\begin{align}
\label{eq42ch5}
\langle DJ_1(f;v^1,v^2),w^2  \rangle = \alpha_1 \iint_{{\mathcal{O}}_{1,d} \times (0,T)} (y-y_{1,d})z\ dxdt + \mu_1 \iint_{\mathcal{O}_1 \times (0,T)} v^1 w^2 dxdt
\end{align}
From $(\ref{eq41ch5})$ and $(\ref{eq42ch5})$ we have
\begin{align}
\label{eq43ch5}
\nonumber
\langle D_1 J_1(f;v^1+sw^1,v^2)- D_1 J_1(f;v^1,v^2),w^2  \rangle &= \alpha_1 \iint_{\mathcal{O}_{1,d} \times (0,T)} (y^s-y_{1,d}) z^s  dxdt \nonumber\\
&- \alpha_1 \iint_{\mathcal{O}_{1,d} \times (0,T)} (y-y_{1,d})z\ dxdt \nonumber\\ 
&+ s \mu_1 \iint_{\mathcal{O}_1 \times (0,T)} w^1 w^2 dxdt.
\end{align}
Let us introduce the adjoint of $(\ref{EC10ch5})$
\begin{equation}
\label{EC11ch5}
\left\{
\begin{array}{rl}
&-\phi^s_t - ((D_1 a(y^s_x,t,x)y_x^s+ a(y_x^s,t,x)) \phi^s_x)_{x} + D_1 F(y^s,y^s_x) \phi^s - (D_2 F(y^s,y^s_x) \phi^2)_x = \alpha_1 (y^s - y_{1,d}) {1}_{\mathcal{O}_{1,d}} \ \text{in} \ Q,\\
&\phi^s(0,t)=\phi^s(L,t) = 0\ \ \text{ in } \ (0,T),\\
&\phi^s(T) = 0\ \ \text{ in }\ \ I.
\end{array}
\right.
\end{equation}
Multiplying $(\ref{EC10ch5})_1$ by $\phi^s$ in $Q$, integrating by parts and replacing $(\ref{EC11ch5})$, we obtain
\begin{align}
\label{eq44ch5}
\nonumber
\iint_Q (z^s_t - ((D_1 a(y^s_x,t,x)y_x^s+ a(y_x^s,t,x)) z^s_x)_{x} &+ D_1F(y^s,y^s_x) z^s + D_2 F(y^s,y^s_x) z_x^s) \phi^s dxdt = \iint_Q w^2 {1}_{\mathcal{O}_1} \phi^s dxdt\nonumber\\
\iint_Q (-\phi^s_t - ((D_1 a(y_x^s,t,x)y_x^s + a(y_x^s,t,x)) \phi^s_x)_{x} &+ D_1 F(y^s,y^s_x) \phi^s - (D_2 F(y^s,y^s_x) \phi^2)_x) z^s dxdt = \iint_Q w^2 \phi^s {1}_{\mathcal{O}_1}  dxdt \nonumber\\
\iint_Q \alpha_1 (y^s - y_{1,d})z^s {1}_{\mathcal{O}_{1,d}} dxdt &= \iint_Q w^2 \phi^s {1}_{\mathcal{O}_1}  dxdt
\end{align}
From $(\ref{eq43ch5})$ and $(\ref{eq44ch5})$ , we have
\begin{align}
\label{eq45ch5}
\langle D_1J_1(f;v^1+sw^1,v^2)- D_1 J_1(f;v^1,v^2),w^2  \rangle &=  \iint_{\mathcal{O}_{1} \times (0,T)} (\phi^s-\phi) w^2  dxdt + s \mu_1 \iint_{\mathcal{O}_1 \times (0,T)} w^1 w^2 dxdt.
\end{align}
Notice that
\begin{align*}
-(\phi^s - \phi)_t &- \Bigl[ \Bigl(D_1 a(y_x^s,t,x)y_x^s + a(y_x^s,t,x) \Bigr)(\phi^s_{x}-\phi_{x}) \Bigr]_x\\
& - \Bigl[ \Bigl((D_1 a(y_x^s,t,x)- D_1 a(y_x,t,x))y_x^s + D_1 a(y_x,t,x)(y^s - y)_x + a(y^s_x,t,x) - a(y_x,t,x) \Bigr)\phi_{x} \Bigr]_x\\
&+ [D_1 F(y^s,y^s_x) - D_1 F(y,y^s_x)] \phi^s + [D_1 F(y,y^s_x)-D_1 F(y,y_x)]\phi^s\\
&+ D_1 F(y,y_x) (\phi^s - \phi) - ([D_2 F(y^s,y^s_x) - D_2 F(y,y^s_x)] \phi^s)_x\\
&- ([D_2 F(y,y^s_x) - D_2 F(y,y_x)] \phi^s)_x - (D_2F(y,y_x)[\phi^s-\phi])_x\\
&=\alpha_1 (y^s-y) {1}_{\mathcal{O}_{1,d}},
\intertext{and}
(y^s-y)_t &- \Bigl[ (a(y_x^s,t,x)-a(y_x,t,x)) y_x^s + a(y_x,t,x)(y_x^s-y_x) \Bigr]_x + [F(y^s,y^s_x) - F(y,y^s_x)] \\
&+ [F(y,y^s_x) - F(y,y_x)] = s w^1 {1}_{\mathcal{O}_1}.
\end{align*}
Consequently, the limits
$$\eta = \lim_{s \to 0} \frac{1}{s}(\phi^s-\phi)\ \ \ \text{and}\ \ \ h=\lim_{s \to 0} \frac{1}{s}(y^s-y)$$
exist and satisfy
\begin{equation}
\label{EC12ch5}
\left\{
\begin{array}{rl}
&-\eta_t - \Bigl[ \Bigl(D_1 a(y_x,t,x)y_x + a(y_x,t,x) \Bigr) \eta_{x} \Bigr]_x - \Bigl[ \Bigl(D_{11}^2 a(y_x,t,x) y_x h_x + 2 D_1 a(y_x,t,x)h_x \Bigr)\phi_{x} \Bigr]_x \\
&+D_{11}^2 F(y,y_x) \phi h + D_{12}^2 F(y,y_x) \phi h_x + D_1 F(y,y_x) \eta\\
&-(D_{21}^2 F(y,y_x) \phi h)_x - (D_{22}^2 F(y,y_x) \phi h_x)_x - (D_2 F(y,y_x) \eta)_x\\
&= \alpha h {1}_{\mathcal{O}_{1,d}} \ \text{in} \ Q,\\
&h_t - \Bigl[ \Bigl(D_1 a(y_x,t,x) y_x + a(y_x,t,x) \Bigr) h_{x} \Bigr]_x + D_1F(y,y_x) h + D_2F(y,y_x) h_x = w^1 {1}_{\mathcal{O}_1}\ \text{in} \ Q,\\
&\eta(0,t)=\eta(L,t) = 0,\ \ h(0,t)=h(L,t)=0\ \ \text{ in } \ (0,T),\\
&\eta(T) = 0,\ \ h(0)=0\ \ \text{ in } \ I.
\end{array}
\right.
\end{equation}
Thus, from $(\ref{eq45ch5})$ and $(\ref{EC12ch5})$, we deduce that
\begin{align*}
\langle D^2_1 J_1(f;v^1,v^2),(w^1,w^2)  \rangle &=  \iint_{\mathcal{O}_1 \times (0,T)} \eta w^2  dxdt +  \mu_1 \iint_{\mathcal{O}_1 \times (0,T)} w^1 w^2 dxdt.
\end{align*}
In particular, for $w^2=w^1$, one has
\begin{align}
\label{eq46ch5}
\langle D^2_1 J_1(f;v^1,v^2),(w^1,w^1)  \rangle &=  \iint_{\mathcal{O}_1 \times (0,T)} \eta w^1  dxdt +  \mu_1 \iint_{\mathcal{O}_1 \times (0,T)} |w^1|^2 dxdt.
\end{align}
Let us show that, for some $C$ only depending on $I$, $\mathcal{O}$, $\mathcal{O}_i$, $T$, $\mathcal{O}_{i,d}$, $\alpha_1$, we have
\begin{align}
\label{eq47ch5}
\left| \iint_{\mathcal{O}_1 \times (0,T)} \eta w^1  dxdt\right| \leq C (1+\|y_0\|+\|f\|_{L^2(\mathcal{O} \times (0,T))}) \|w^1\|^2_{L^2(\mathcal{O}_1 \times (0,T))}
\end{align}
We also get the following
\begin{align}
\label{eq49ch5}
\nonumber
\iint_{\mathcal{O}_1 \times (0,T)} &\eta w^1  dxdt\nonumber\\
= \iint_Q &(h_t - [(D_1 a(y_x,t,x)y_x+a(y_x,t,x))h_{x}]_x + D_1 F(y,y_x) h + D_2 F(y,y_x) h_x) \eta dxdt \nonumber\\
= \iint_Q &h(-\eta_t - [(D_1 a(y_x,t,x)y_x + a(y_x,t,x))\eta_{x}]_x + D_1 F(y,y_x) \eta + (D_2  F(y,y_x) \eta)_x dxdt \nonumber\\
= \iint_Q &h([(D_{11}^2 a(y_x,t,x) y_x h_x + 2 D_1 a(y_x,t,x)h_x )\phi_{x}]_x - D_{11}^2 F(y,y_x) \phi h - D_{12}^2 F(y,y_x) \phi h_x \nonumber\\
&+ (D_{21}^2 F(y,y_x) \phi h)_x + (D_{22}^2 F(y,y_x) \phi h_x)_x + \alpha h {1}_{\mathcal{O}_{1,d}} ) dxdt \nonumber\\
= \iint_Q &(D_{11}^2 a(y_x,t,x) |h_x|^2 y_x \phi_x + 2 D_1 a(y_x,t,x) |h_x|^2 \phi_x - D_{11}^2 F(y,y_x) \phi |h|^2 - D_{12}^2 F(y,y_x) \phi h_x h \nonumber\\
&- D_{21}^2 F(y,y_x) \phi h h_x + D_{22}^2 F(y,y_x) \phi |h_x|^2 + \alpha |h|^2 {1}_{\mathcal{O}_{1,d}}) dxdt \nonumber\\
\leq C\ & \left( \int_0^T \|h_{xx}(t)\|^2 \|y_x(t)\| \|\phi_x(t)\| dt + \int_0^T \|h_{xx}(t)\|\|h_x(t)\| \|\phi_x(t)\| dt \right. \nonumber\\
&\left. + \int_0^T \|h_x(t)\| \|h(t)\| \|\phi_x(t)\| + \iint_Q |h|^2 dxdt \right)
\end{align}
From $(\ref{EC11ch5})$ with $s=0$, using energy estimates, we have
\begin{align}
\label{eq50ch5}
\|\phi_{xx}\|^2_{L^2(Q)} + \|\phi_x\|^2_{L^\infty(0,T;L^2(I))} &\leq C (\|y\|^2_{L^2(Q)}+\|y_{1,d}\|^2_{L^2(\mathcal{O}_{1,d} \times (0,T))})\\
\label{eq51ch5}
\|h_{xx}\|^2_{L^2(Q)} + \|h_x\|^2_{L^\infty(0,T;L^2(I))} &\leq C \|w^1\|^2_{L^2(\mathcal{O}_1 \times (0,T))}
\end{align} 
as $(v^1,v^2)$ is the Nash quasi-equilibrium, then $y$ have the following regularity
\begin{align}
\label{eq52ch5}
\|y\|_{L^2(Q)}^2 \leq C (\underset{[0,T]}{\text{sup}}\|y_x(t)\|^2 + \|y_{xx}\|_{L^2(Q)}^2) \leq C (\|f\|_{L^2(\mathcal{O}\times (0,T))}^2 + \|y_0\|^2+ \sum_{i=1}^2 \frac{1}{\mu}\|\phi^i\|_{L^2(\mathcal{O}_i \times (0,T))}^2)
\end{align}
Using $(\ref{eq49ch5})$ - $(\ref{eq52ch5})$, we have
\begin{align*}
\left|\iint_{\mathcal{O}_1 \times (0,T)} \eta w^1  dxdt\right| 
\leq C\ &(\|\phi_{x}\|_{L^\infty(0,T;L^2)} \|y_x\|_{L^\infty(0,T;L^2)} + \|\phi_{x}\|_{L^\infty(0,T;L^2)}+ 1)\\
&\cdot(\|h_{xx}\|_{L^2} + \|h_{xx}\|_{L^2} \|h_x\|_{L^\infty(0,T;L^2)})\\
\leq C\ &(\|f\|_{L^2(\mathcal{O}\times (0,T))}^2+\|y_0\|^2+ \|f\|_{L^2(\mathcal{O}\times (0,T))}+\|y_0\|+1) \|w^1\|^2_{L^2(\mathcal{O}_1 \times (0,T))}
\end{align*}
This prove $(\ref{eq47ch5})$ in this case.\\
Taking into account $(\ref{eq46ch5})$ and $(\ref{eq47ch5})$, we see that
\begin{align*}
\langle D_1^2 J_1(f;v^1,v^2),(w^1,w^1)  \rangle &\geq \Bigl[\mu_1 - C (\|f\|_{L^2(\mathcal{O}\times (0,T))},\|y_0\|) \Bigr] \iint_{\mathcal{O}_1 \times (0,T)} |w^1|^2 dxdt.
\end{align*}
Note that the previous constant $C$ can be chosen independent of $\mu_1$ and $\mu_2$. It is clear that, for sufficiently large $\mu_1$ and $\mu_2$, the couple $(v^1,v^2)$ is a Nash equilibrium in the sense of Definition $\ref{def1ch5}$.

\section{Hierarchical Controllability with Trajectories}\label{s6}

If let us fix an uncontrolled trajectory of (\ref{EC1ch5}), that is, a sufficiently regular solution to the system
\begin{equation}
\label{trajectch5}
\left\{
\begin{array}{rl}
&\overline{y}_t -  (a(\overline{y}_x,t,x) \overline{y}_x)_x + F(\overline{y},\overline{y}_x) = 0 \ \text{in} \ Q,\\
&\overline{y}(0,t)=\overline{y}(L,t) = 0\ \ \text{ in } \ (0,T),\\
&\overline{y}(0) = \overline{y}_{0}\ \ \text{in } \ I.
\end{array}
\right.
\end{equation}
Once the Nash equilibrium (see Definition \ref{def1ch5}) has been identified and fixed for each $f$, we look for a control $\hat{f} \in L^2(\mathcal{O} \times (0,T))$ subject to the restriction of null controllability
\begin{equation*}
y(T)= \overline{y}(T) \ \ \text{in}\ I.
\end{equation*}
We will have analogous results to Theorem \ref{teo1ch5} and 
Theorem \ref{teo2ch5}.\\
To prove this, we denote $z:= y - \overline{y}$, $z_{i,d}:=y_{i,d}-\overline{y}$ and $z_0 :=y_0 - \overline{y}_0$, obtaining the following equivalent optimality system
\begin{equation}
\label{eq53ch5}
\left\{
\begin{array}{rl}
&z_t - \Bigl(a(z_x + \overline{y}_x,t,x) z_x \Bigr)_x - \Bigl(a(z_x + \overline{y}_x,t,x)-a(\overline{y}_x,t,x) \Bigr)_x + F(z + \overline{y},z_x + \overline{y}_x) - F(\overline{y},\overline{y}_x) \\
&= f 1_\mathcal{O} - \dfrac{1}{\mu_1} p^1 1_{\mathcal{O}_1} - \dfrac{1}{\mu_2} p^2 1_{\mathcal{O}_2} \ \text{in} \ Q,\\
&-p^i_t - \Bigl( \Bigl(D_1 a(z_x +\overline{y}_x,t,x)(z_x+\overline{y}_x) + a(z_x+\overline{y}_x,t,x) \Bigr) p^i_x \Bigr)_{x} + D_1 F(z+\overline{y},z_x+\overline{y}_x) p^i  \\
&- (D_2 F(z+ \overline{y},z_x+\overline{y}_x) p^i)_x = \alpha_i (z-z_{i,d}) 1_{\mathcal{O}_{i,d}} \ \text{in} \ Q,\\
&z(0,t) = z(L,t) = 0, \ \  p^i(0,t)=p^i(L,t)=0\ \ \text{ in } \ (0,T),\\
&z(0) = z_{0}, \ \  p^i(T)=0\ \ \text{in } \ I.
\end{array}
\right.
\end{equation}
Now, we will study the Null Controllability for the state $z$ in the system (\ref{eq53ch5}), this is
\begin{equation*}
z(T) = 0, \ \ \text{in}\ I.
\end{equation*}
Using the classical techniques, we must study the Null Controllability for the linearized system at $0$
\begin{equation}
\label{eq54ch5}
\left\{
\begin{array}{rl}
&z_t - \Bigl(\Bigl(D_1 a(\overline{y}_x,t,x)\overline{y}_x + a(\overline{y}_x,t,x) \Bigr) z_{x} \Bigr)_x + D_1 F(\overline{y},\overline{y}_x)z + D_2F(\overline{y},\overline{y}_x) z_x \\
&=  f 1_\mathcal{O} - \dfrac{1}{\mu_1} p^1 1_{\mathcal{O}_1} - \dfrac{1}{\mu_2} p^2 1_{\mathcal{O}_2} + G\ \text{in} \ Q,\\
&-p^i_t - \Bigl( \Bigl(D_1 a(\overline{y}_x,t,x) \overline{y}_x + a(\overline{y}_x,t,x) \Bigr) p^i_{x} \Bigr)_x + D_1 F(\overline{y},\overline{y}_x) p^i - (D_2F(\overline{y},\overline{y}_x) p^i)_x \\
&= \alpha_i z 1_{\mathcal{O}_{i,d}} + G_i\ \text{in} \ Q,\\
&z(0,t) = z(L,t) = 0, \ \  p^i(0,t)=p^i(L,t)=0\ \ \text{ in } \ (0,T),\\
&z(0) = z_{0},\ \ p^i(T)=0, \ \ \text{ in } \ I.
\end{array}
\right.
\end{equation}
In this case, the trajectory have an additional condition
\begin{equation}
	\label{traject-condition}
	\|\overline{y}_x\|_{L^\infty(I \times (0,T))} \leq \frac{a_0}{2M}
\end{equation}
where the constants $a_0$ and $M$ were defined in the Section \ref{s2}. \\
Then to use Carleman estimates, we consider the adjoint system for $(\ref{eq54ch5})$
\begin{equation}
\label{eq55ch5}
\left\{
\begin{array}{rl}
&-\varphi_t -  \Bigl( \Bigl(D_1 a(\overline{y}_x,t,x)\overline{y}_x + a(\overline{y}_x,t,x) \Bigr) \varphi_{x} \Bigr)_x + D_1 F(\overline{y},\overline{y}_x) \varphi - (D_2F(\overline{y},\overline{y}_x) \varphi)_x=  \alpha_1 \theta^1 1_{\mathcal{O}_{1,d}} + \alpha_2 \theta^2 1_{\mathcal{O}_{2,d}} + \mathcal{G} \ \text{in} \ Q,\\
&\gamma^i_t - \Bigl( \Bigl(D_1 a(\overline{y}_x,t,x) \overline{y}_x + (a(\overline{y}_x,t,x) \Bigr)  \gamma^i_{x} \Bigr)_x + D_1 F(\overline{y},\overline{y}_x) \gamma^i + D_2 F(\overline{y},\overline{y}_x) \gamma^i_x =  - \dfrac{1}{\mu_i} \varphi 1_{\mathcal{O}_i} + \mathcal{G}_i\ \text{in} \ Q,\\
&\varphi(0,t) = \varphi(L,t) = 0, \ \ \gamma^i(0,t)=\gamma^i(L,t) =0, \ \ \text{ in } \ (0,T),\\
&\varphi(T) = \varphi^{T},\ \ \gamma^i(0)=0, \ \ \text{ in } \ I.
\end{array}
\right.
\end{equation}
We conclude the proof making similar counts of the Section \ref{s4} and Section \ref{s5}. 

\section{Some additional comments and open questions}
\label{s7}
\begin{itemize}
	\item[i)] The hierarchic controllability for the system (\ref{EC1ch5}) in higher dimension is an open problem. The chose of the suitable Banach spaces to guarantee the Liusternik's Theorem is a difficult task. Furthermore, to guarantee the additional estimate we need strongly the fact of the dimension is 1, because we use the embedding $H^1(I) \hookrightarrow L^{\infty}(I)$ and this result in higher dimension is not true. A good advance in this direction would be to follow the ideas in \cite{XU}, that is to say, to prove the controllability using Kakutani Fixed Point Theorem.
	
	\item[ii)] Following the line of study for this paper, a coupled system with the same nonlinearity that the system (\ref{EC1ch5}) is a very interesting problem. In fact, in the system
	 \begin{equation*}
\left\{
\begin{array}{rl}
&y_{1,t} -  \nabla \cdot (a(\nabla y_1,t,x)\nabla y_1) + F_1(y_1,y_2,\nabla y_1,\nabla y_2) = f 1_\mathcal{O} + v^1 1_{\mathcal{O}_1} + v^2 1_{\mathcal{O}_2} \ \text{in} \ Q,\\
&y_{2,t} -  \nabla \cdot (a(\nabla y_2,t,x)\nabla y_2) + F_2(y_1,y_2,\nabla y_1,\nabla y_2) = 0 \ \text{in} \ Q,\\
&y(x,t) = 0\ \ \text{ in } \ \partial \Omega \times (0,T),\\
&y(0) = y_{0}\ \ \text{in } \ \Omega.
\end{array}
\right.
\end{equation*}
the hierarchic controllability is an open problem.

The complication is in the Observability inequality. A good possible solution would be change the functional of the Nash equilibrium, putting suitable weight functions in the follower's spaces, this is, we assume the functional
$$J_i(f;v^1,v^2) = \frac{\alpha_i}{2} \iint_{\mathcal{O} \times (0,T)} (|y_1-y_{1,d}|^2 + |y_2-y_{2,d}|^2)dxdt + \frac{\mu_i}{2} \iint_{\mathcal{O}_i} {\rho^*}^2 |v^i|^2 dxdt$$
where $\rho^* \geq e^{s\sigma_j/2}$ for $j=1,2$  (see \cite{Tere}).

	\item[iii)] We consider the degenerate parabolic system
	\begin{equation}
	\label{EC13ch5}
\left\{
\begin{array}{rl}
&y_t -  (b(y_x,t) x^{\alpha}  y_x)_x + F(y,y_x) = f 1_\mathcal{O} + v^1 1_{\mathcal{O}_1} + v^2 1_{\mathcal{O}_2},\\
&y(1,t) = 0\ \ \text{and} \begin{cases}
y(0,t) = 0\ \ &\text{for}\ 0 \leq \alpha <1,\\
(x^{\alpha} y_x)(0,t) = 0\ \ &\text{for}\ 1\leq \alpha <2,
\end{cases}\\
&y(0) = y_{0}.
\end{array}
\right.
\end{equation}
There are many papers about the controllability for the system linearized (see \cite{Canar1}, \cite{Canar2}), but the hierarchic controllability in (\ref{EC13ch5}) is an open problem. It is not possible applied similar techniques of this paper, because the function $b(s,t)x^\alpha$ does not satisfy the conditions 2 and 3 to the function $a(s,x,t)$ in Section~\ref{s2}. The study of this problem is a future work for us.

\end{itemize}



\end{document}